\documentclass[11pt]{article}%
\usepackage{graphicx}
\usepackage{amsmath}
\usepackage{amsfonts}
\usepackage{amssymb}
\usepackage{color}%
\providecommand{\U}[1]{\protect\rule{.1in}{.1in}}
\newtheorem{theorem}{Theorem}
\newtheorem{acknowledgement}[theorem]{Acknowledgement}

\newtheorem{corollary}[theorem]{Corollary}

\newtheorem{definition}[theorem]{Definition}

\newtheorem{lemma}[theorem]{Lemma}

\newtheorem{proposition}[theorem]{Proposition}
\newtheorem{remark}[theorem]{Remark}

\newenvironment{proof}{\paragraph{\textcolor{black}{Proof:}}}{\hfill\textcolor{black}{$\square$}}
\begin{document}
\title{Existence of a maximal solution of singular parabolic equations with absorptions:  quenching phenomenon and the instantaneous shrinking phenomenon }
\author{Nguyen Anh Dao\\ Faculty of Mathematics and Statistics,\\ Ton Duc Thang University, Ho Chi Minh City, Vietnam
\\
daonguyenanh@tdt.edu.vn}

\maketitle
\tableofcontents
\bigskip
\textbf{Abstract.} This paper deals with nonnegative solutions of the one dimensional degenerate parabolic equations  with zero homogeneous Dirichlet boundary condition.  To obtain an existence result, we prove a sharp gradient estimate of $|u_x|$. 
 Besides, we investigate the behaviors of nonnegative solutions such as the quenching phenomenon, and the finite speed of propagation.
 Our results of the Dirichlet problem will be extended to the associated Cauchy problem. In addition, we show that the phenomenon of the instantaneous shrinking of compact support of the nonnegative solutions occurs if $f$ satisfies some growth condition.
\medskip

\noindent\textbf{Mathematics Subject Classification (2000):} 35K55, 35K65, 35B99.

\noindent Key words: \textit{gradient estimates, quenching type of parabolic
equations, irregular initial datum, free
boundary, instantaneous shrinking of compact support.}

\section{Introduction}

\hspace{0.2in} In this paper, we study the nonnegative solutions of the one dimensional degenerate parabolic equation  on a
given open bounded interval $I=(-l,l)$
\begin{equation}\label{plap1}
\left\{
\begin{array}
[c]{lr}%
\partial_{t}u-(|u_{x}|^{p-2}u_{x})_{x}+u^{-\beta}\chi_{\{u>0\}}  + f(u)=0 &
\text{in}\hspace{0.05in}I\times(0,\infty),\\
u(-l,t)=u(l,t)=0 & t\in(0,\infty),\\
u(x,0)=u_{0}(x) & \hspace{0.05in}\text{in}\hspace{0.05in}I,
\end{array}
\right.  
\end{equation}
where $\beta\in(0,1)$, $p>2$;
and $\chi_{\{u>0\}}$ denotes the
characteristic function of the set of points $(x,t)$ where $u(x,t)>0$, i.e 
\[
\chi_{\{u>0\}}= \left\{
\begin{array}
[c]{lr}%
1,& \text{ if } u>0, \\
0, & \text{ if } u\leq 0.
\end{array}
\right.  
\] 
Note that the absorption term $\chi_{\{u>0\}}u^{-\beta}$ becomes singular when $u$ is near to $0$, and we impose $\chi_{\{u>0\}}u^{-\beta}=0$ whenever $u=0$. Through this paper, we assume that $f: \mathbb{R}\rightarrow \mathbb{R}$, $f\in\mathcal{C}(\mathbb{R})$ is a nonnegative  function. But,  $f$ will be addressed in detail later for the existence of solution, see $(H_1)$ and $(H_2)$ below.  
\\

As already known,  problem (\ref{plap1}) in the semi-linear case ($p=2$, and $f=0$) can be considered as a limit of mathematical  models arising in Chemical Engineering corresponding to catalyst kinetics  of
Langmuir-Hinshelwood type (see, e.g. \cite{Strieder-Aris} p. 68, \cite{Phillips} and reference therein). The semi-linear case was studied in many papers such as \cite{Phillips}, \cite{Kawohl}, \cite{Levine}, \cite{Davila-Montenegro survey}, \cite{AnhDiazPaul}, \cite{Winkler-non uniquenes}, and so forth. These papers focused on studying the existence of solution, and the behaviors of solutions. From our knowledge, the existence result of the semi-linear case  was first proved by Phillips  for the Cauchy problem (see Theorem $1$, \cite{Phillips}). The same result holds for the semi-linear equation with positive Dirichlet boundary condition (see Theorem $2$, \cite{Phillips}).  Moreover, he proved a property of the finite speed of propagation of nonnegative solutions, i.e, any solution with compact support initially has compact support at all later times  $t>0$.  
\\

The semi-linear problem of this type was also extended in many aspects. In \cite{Davila-Montenegro survey}, J. Davila, and M. Montenegro proved the existence of solution with zero Dirichlet boundary condition  with a source term $f(u)$. We emphasize that the equations of this type with zero Dirichlet boundary condition are harder than the one of positive  Dirichlet boundary condition because of the effect of the singular term $u^{-\beta}\chi_{\{u>0\}}$.  Furthermore, they showed that the uniqueness result holds for a particular class of positive solutions, see Theorem $1.10$ in \cite{Davila-Montenegro survey}. Recently, Diaz et al., \cite{AnhDiazPaul}, proved a uniqueness result for a   class of solutions, which is different from the one of \cite{Davila-Montenegro survey}.  However,  Winkler  showed that the uniqueness result fails  in general (see Theorem $1.1$, \cite{Winkler-non uniquenes}). 
\\

After that, the equations of this type was considered under  more general forms. For example, the case of quasilinear diffusion operators was already considered in \cite{Kawohl} (for a different diffusion term). We also mention here  the 
porous medium of this type was studied by B. Kawohl and R. Kersner,  \cite{Kaw-Kersner}. We note that problem
$(\ref{plap1})$ was considered recently by Giacomoni et al., \cite{GiSaSer} with $f(u)$ on the right hand side, but there was  a technical fault in the proof of the existence of  solution. 
\\

Inspired by the above studies, we would like to investigate the existence of nonnegative solutions and the behaviors of solutions  of  equation $(\ref{plap1})$. Before stating our
main results, let us define the notion of a weak solution of equation
$(\ref{plap1})$. 
\begin{definition}
Given $0\leq u_{0}\in L^{1}(I)$. A function $u\geq 0$ is called a weak solution of equation
$(\ref{plap1})$ if $f(u), u^{-\beta}\chi_{\{u>0\}}\in L^{1}(I\times(0,\infty))$, and $u\in L_{loc}^{p}(0,\infty;W_{0}^{1,p}(I))\cap
L_{loc}^{\infty}(\overline{I}\times(0,\infty))\cap\mathcal{C}([0,\infty
);L^{1}(I))$ satisfies equation
$(\ref{plap1})$ in
the sense of distributions $\mathcal{D^{\prime}}(I\times(0,\infty))$, i.e,
\begin{equation}\label{plapdef}
\int_{0}^{\infty}\int_{I}-u\phi_{t}+|u_{x}|^{p-2}u_{x}\phi_{x}+u^{-\beta}\chi
_{\{u>0\}}\phi  + f(u)\phi    \hspace{0.05in} dxdt=0,\quad\forall\phi\in
\mathcal{C}_{c}^{\infty}(I\times(0,\infty)). 
\end{equation}
\end{definition}
\bigskip

Next, if  $f$ satisfies either $(H_1)$ or $(H_2)$ below, we have then an existence of solution of problem $(\ref{plap1})$.
\begin{align*}
& (H_1)\quad   f\in \mathcal{C}^1(\mathbb{R}) \text{ and } f(0)=0.
\\
& (H_2)\quad   f \text{ is a nondecreasing function, and } f(0)=0.
\end{align*}
\begin{theorem}\label{theplapexistence1} Let $0\leq u_{0}\in L^{\infty}(I)$, and $f$ satisfy $(H_1)$. Then, there exists a maximal weak bounded solution $u$ of equation $(\ref{plap1})$. Moreover, we have 
\\

There exists a positive constants $C(\beta,p)$  
such that  
\begin{equation}\label{plap3d}
| u_x (x,t)| \leq
 C . u^{1-\frac{1}{\gamma}} (x, t) \left(  t^{-\frac{1}{p}}  \|u_{0}\|_{\infty}^{\frac{1+\beta}{p}} + M_f(u_0).\|u_{0}\|_{\infty}^{\frac{\beta}{p}} + M_{f^\prime}(u_0).\|u_{0}\|_{\infty}^{\frac{1+\beta}{p}} + 1 \right), 
\end{equation}
for a.e $(x,t)\in I\times(0,\infty)$, with  $\gamma=\displaystyle\frac{p}{p+\beta-1}$,  and 
  $M_g(u_0)  =  \displaystyle \left(\max_{0 \leq s \leq 2\|u_0\|_{\infty}}  |g(s)| \right)^{\frac{1}{p}}$, for any $g\in\mathcal{C}(\mathbb{R})$.
\\

As a consequence of $(\ref{plap3d})$,  for any $\tau>0$
 there is a positive constant $C(\beta
,p,\tau,\Vert u_{0}\Vert_{\infty})$ such that
\begin{equation}\label{plap3f}
|u(x,t)-u(y,s)|\leq C\left(  |x-y|+|t-s|^{\frac{1}{2}}\right)  ,\quad\forall
x,y\in\overline{I},\quad\forall t,s \geq \tau. 
\end{equation}
\end{theorem}

\begin{theorem}
\label{theplapexistence} Let $0\leq u_{0}\in L^{1}(I)$, and $f$ satisfy $(H_2)$. Then, there exists a maximal weak solution $u$ of equation $(\ref{plap1})$. Furthermore, we have 
\\

 For any $\tau>0$, there exist two positive constants $C_1(\beta,p,|I|)$ and $C_2(p, |I|)$ 
such that  
\begin{equation}
|u_x(x,t)|\leq C_1 . u^{1-\frac{1}{\gamma}}(x,t)\left(  \tau^{-\frac{\lambda+\beta+1}{\lambda p}%
}\Vert u_{0}\Vert_{L^{1}(I)}^{\frac{1+\beta}{\lambda}} +  \tau^{-\frac{\beta}{\lambda p}}\Vert u_{0}\Vert_{L^{1}(I)}^{\frac{\beta}{\lambda}} . m_f(\tau, u_0)    +1     \right), \label{plap3b}
\end{equation}
for a.e $(x,t)\in I\times(\tau,\infty)$, with $\lambda=2(p-1)$,  and 
  $m_f(\tau, u_0)= f^\frac{1}{p}\left(C_2 . \tau^{-\frac{1}{\lambda}}\|u_0\|^{\frac{p}{\lambda}}_{L^1(I)} \right)$.
\\

As a consequence of $(\ref{plap3b})$, 
 there is a positive constant $C(\beta
,p,\tau,|I|,\Vert u_{0}\Vert_{L^{1}(I)})$ such that
\begin{equation}
|u(x,t)-u(y,s)|\leq C\left(  |x-y|+|t-s|^{\frac{1}{2}}\right)  ,\quad\forall
x,y\in\overline{I},\quad\forall t,s \geq \tau. \label{plap3c}%
\end{equation}
\end{theorem}
\begin{remark}
Note that estimate $(\ref{plap3b})$ does not include the term of $f^\prime$, compare with $(\ref{plap3d})$.  Actually, this one is a combination of estimate $(\ref{plap3d})$ without $M_{f^\prime}(u_0)$, and the smoothing effect $L^1-L^\infty$.
\end{remark}
\begin{remark}
Conclusion $(\ref{plap3f})$ (resp. $(\ref{plap3c})$) implies that $u$ is continuous up to the boundary.
This result answers an open question stated in the Introduction of
\cite{Winkler-non uniquenes} for the semi-linear case.
\end{remark}
\begin{remark}
When $p=2$ and $f=0$, estimate $(\ref{plap3d})$ becomes the gradient estimates in \cite{Phillips}, \cite{Davila-Montenegro survey}, \cite{Winkler-non uniquenes}.  
\end{remark}
\begin{remark}
The condition $f(0)=0$ in $(H_1)$ and $(H_2)$ is necessary for the existence of nonnegative solutions. If $f$ violates this one, i.e, $f(0)>0$ then the existence result  fails, see Corollary \ref{Cornonexist}.
\end{remark}

A second goal of this article is to study the most striking phenomenon of equations of  this type, the so called  quenching phenomenon that solution vanishes after a finite time. This property arises due to the presence of the singular term $u^{-\beta}\chi_{\{u>0\}}$. It occurs even starting with a positive unbounded  initial data and there is a lack of uniqueness of solutions (see Theorem $1.1$, \cite{Winkler-non uniquenes} again). Then we have the following results
\begin{theorem}
\label{theplapqueching1} Assume as in Theorem $\ref{theplapexistence1}$. Let $v$ be
any weak solution of equation $(\ref{plap1})$. Then, there is a finite time
$T_{0}=T_{0}(\beta, p, \Vert u_{0}\Vert_{\infty})$ such that
\[
v(t)=0,\quad\text{for }t\geq T_{0}.
\]
\end{theorem}
\begin{theorem}
\label{theplapqueching} Assume as in Theorem $\ref{theplapexistence}$. Let $v$ be
any weak solution of equation $(\ref{plap1})$. Then, there is a finite time
$T_{0}=T_{0}(\beta,p,|I|,\Vert u_{0}\Vert_{L^{1}(I)})$ such that
\[
v(t)=0,\quad\text{for }t\geq T_{0}.
\]
\end{theorem}

Besides, we shall investigate the existence of solution of the Cauchy problem associated to equation $(\ref{plap1})$.
\begin{equation} \label{plapCauchy}
\left\{
\begin{array}
[c]{lr}%
\partial_{t}u-(|u_{x}|^{p-2}u_{x})_{x}+ u^{-\beta}\chi_{\{u>0\}}+ f(u)=0, &
\text{in}\hspace{0.05in}\mathbb{R}\times(0,\infty),
\\
u(x,0)=u_{0}(x), & \hspace{0.05in}\text{in}\hspace{0.05in}\mathbb{R}.
\end{array}
\right. 
\end{equation} 
Moreover, we also study behaviors of solutions  of Cauchy problem such as the quenching phenomenon, and the finite speed of propagation.  In particular, we show that if $f$ satisfies a certain growth condition at infinity, then any weak solution has the instantaneous shrinking of compact support (in short ISS), namely, if $u_0$ only goes to $0$ uniformly as $|x|\rightarrow\infty$, then the support of any weak  solution  is bounded for any $t>0$.  Concerning the ISS phenomenon,  we refer to \cite{BorU}, \cite{EvanKnerr}, \cite{Herrero}, and reference therein. Then, our main result of the Cauchy problem is as follows
\begin{theorem}
\label{theexistenceCauchy1} Let $0\leq u_{0}\in
L^{1}(\mathbb{R})\cap L^{\infty}(\mathbb{R})$. Assume that $f$ satisfies either $(H_1)$ or $(H_2)$. Then, there exists a weak bounded solution
$u\in\mathcal{C}([0,\infty);L^{1}(\mathbb{R}))\cap L^p(0,T; W^{1,p}(\mathbb{R}))$, satisfying
equation $(\ref{plapCauchy})$ in $\mathcal{D^{\prime}}(\mathbb{R}%
\times(0,\infty))$. 
\\

i) Furthermore, any solution with compact support initially has compact support for any $t>0$. And, the solution $u$ constructed above is a maximal solution of equation $(\ref{plapCauchy})$.
\\

ii) In addition, if $u_0(x)\rightarrow 0$ uniformly as $x\rightarrow\infty$, and   $f$ satisfies the following  growth condition at infinity:
 \[
 (H_3)  \quad\text{There is a real number $q_0\in(0,1)$ such that }  f(s)  \geq s^{q_0},\text{ when } s\rightarrow +\infty,
 \]
then such a weak solution of problem $(\ref{plapCauchy})$ has ISS property.
\end{theorem}
\begin{remark}
We note that  our results above also hold for the case where $f$ is only a global Lipschitz function with $f(0)=0$, see Remark \ref{rem20}, Remark \ref{remfinal}, and Theorem \ref{thegloLip}.
\end{remark}

The paper is organized as follows: Section $2$ is devoted to prove 
 a sharp gradient estimate, which is the main key of proving the
existence of solution. In section $3$, we shall give the  proof of
Theorem \ref{theplapexistence}, and  Theorem \ref{theplapexistence1} is proved  in the same way. Section $4$ is devoted to study the quenching phenomenon (including the proofs of Theorem
\ref{theplapqueching} and Theorem \ref{theplapqueching1}). Finally, Section 5 concerns studying  the existence of solution of the associated Cauchy problem, and  behaviors of solutions, thereby includes the proof of Theorem \ref{theexistenceCauchy1}.
\\

Several notations which will be used through this paper are the following: we
denote by $C$ a general positive constant, possibly varying from line to line.
Furthermore, the constants which depend on parameters will be emphasized by
using parentheses. For example, $C=C(p,\beta,\tau)$ means that $C$ only
depends on $p,\beta,\tau$. We also denote by $I_{r}(x)=(x-r,x+r)$ to the open ball
 with center at $x$ and radius $r>0$ in $\mathbb{R}$. If $x=0$, we denote $I_r(0)=I_r$.
Next $\partial_x u$ (resp. $\partial_t u$) means the partial derivative with respect to $x$ (resp. $t$). We also write $\partial_x u=u_x$. Finally, the $L^\infty$-norm of $u$  is denoted by $\|u\|_{\infty} $.
\begin{acknowledgement}
This research  was supported by the ITN FIRST of the Seventh
Framework Program of the European Community (grant agreement number 238702).
\end{acknowledgement}

\section{A sharp gradient estimate}
\hspace{0.2in}In this part, we shall modify Bernstein's technique to obtain estimates on $|u_{x}|$, so called the gradient estimate in $N$-dimension.  Roughly speaking, the gradient estimates that we shall prove are of the type
\begin{equation}\label{plap4b}
|u_x (x,t)|\leq C_1 . u^{1-\frac{1}{\gamma}}(x,t)\left(1  + C_2(f, f^\prime)\right),\quad\text{for a.e
}(x,t)\in I\times(0,\infty),
\end{equation} 
where the constant $C_1$ merely depends on the parameters $\beta, p$, while  $C_2$ involves the terms of $f$ and $f^\prime$.  
It is
well known that such a  gradient estimate  of  $(\ref{plap4b})$ plays a crucial role in
proving the existence of solution (see, e.g. \cite{Phillips},
\cite{Davila-Montenegro survey}, \cite{Winkler-non uniquenes} for the semi-linear case; and see \cite{Kaw-Kersner} for the  porous medium of this type). The degeneracy of the diffusion operator as
$p>2$ leads, obviously, to a considerable amount of additional technical
difficulties.
In the case $f=0$, it is not difficult to show that  estimate $(\ref{plap4b})$
 becomes an equality for a suitable constant $C_1$ ( $C_2=0$), when considering the stationary equation of $(\ref{plap1})$. 
That is the reason why   such a gradient estimate of this type is called 
\textit{a sharp gradient estimate} (since  the power of $u$ in $(\ref{plap4b})$ cannot bigger or smaller than $1-1/\gamma$). By the appearance of the nonlinear diffusion, $p$-laplacian,  we shall establish previously the gradient estimates for
the solutions of the following regularizing problem.
\\

For any $\varepsilon>0$, let us set
\[
g_{\varepsilon}(s)=s^{-\beta}\psi_{\varepsilon}(s),\text{ with }%
\psi_{\varepsilon}(s)=\psi(\frac{s}{\varepsilon}),
\]
 and $\psi\in\mathcal{C}^{\infty}(\mathbb{R})$, $0\leq\psi\leq1$ is a
non-decreasing function such that 
$\psi(s)=\left\{
\begin{array}
[c]{lr}%
0, & \text{if }s\leq1,\\
1, & \text{if }s\geq2.
\end{array}
\right.$
\\
Now fix  $\varepsilon>0$,  we consider the following
problem
\begin{equation}
(P_{\varepsilon, \eta})\left\{
\begin{array}
[c]{lr}%
\partial_{t}z-(a(z_{x})z_{x})_{x}+g_{\varepsilon}(z)  + f(z)\psi_\varepsilon(z) =0, & \text{in}%
\hspace{0.05in}I\times(0,\infty),\\
z(-l,t)=z(l,t)=\eta, & \hspace{0.05in}t\in(0,\infty),\\
z(x,0)=z_{0}(x)+\eta, & \hspace{0.05in}x\in I,
\end{array}
\right.  \label{plap4}%
\end{equation}
with 
$a(s)=b(s)^{\frac{p-2}{2}} $, $b(s)=|s|^{2}+\eta^{\alpha}$;  
$\alpha>0$ will be addressed  later; and $\eta\rightarrow 0^+$.
Note that 
 $a(z_{x})$ is a regularization of $|z_{x}|^{p-2}$. Then, problem $(P_{\varepsilon, \eta})$ can be understood as a regularization of
equation $(\ref{plap1})$. The gradient estimates, presented in
this framework are as follows:

\begin{lemma}
\label{lemgradientestimate} 
Given  $0\leq z_{0} \in\mathcal{C}_{c}^{\infty
}(I),\hspace{0.05in} z\not=0$. Assume that $f\in \mathcal{C}^1(\mathbb{R})$ is a nonnegative function. Then,  for any $\eta\in(0,\displaystyle\min\{\varepsilon, \|z_0\|_{L^{\infty}(I)}\})$, there
exists a unique classical solution $z_{\varepsilon, \eta}$ of equation
$(\ref{plap4})$. Moreover, there is a positive constant $C(\beta, p)$ such
that 
\begin{equation}\label{plapgradient}
|\partial_{x} z_{\varepsilon, \eta}(x,\tau)| \leq
 C . z_{\varepsilon, \eta}^{1-\frac{1}{\gamma}} (x, \tau) \left(  \tau
^{-\frac{1}{p}}  \|z_{0}\|_{\infty}^{\frac{1+\beta}{p}} +    M_f(z_0) .\|z_{0}\|_{\infty}^{\frac{\beta}{p}}   + M_{f^\prime}(z_0) .\|z_{0}\|_{\infty}^{\frac{1+\beta}{p}}+ 1 \right),
\end{equation}
for $(x,\tau) \in I\times(0,\infty)$. Recall here \hspace{0.05in} $M_g(u_0)  =  \displaystyle \left(\max_{0 \leq s \leq 2\|u_0\|_{\infty}}  |g(s)| \right)^{\frac{1}{p}}$.
\end{lemma}
\begin{proof}
The existence and uniqueness of solution, $z_{\varepsilon,\eta}
\in\mathcal{C}^{\infty}(\overline{I}\times\lbrack0,\infty))$ is well-known (see, e.g.  \cite{Herrero}, \cite{LaSoU}, \cite{WuJingHui}, \cite{Herrero} and
\cite{Zhao}). For sake of brevity, let us drop dependence on $\varepsilon
,\eta$ in the notation  of $z_{\varepsilon, \eta}$, and put 
\[
z=z_{\varepsilon,\eta}.
\]

It is clear that $\eta$ (resp. $\Vert z_{0}\Vert_{L^{\infty}%
(I)}+\eta$) is a sub-solution (resp. super-solution) of equation
$(\ref{plap4})$. Then, the comparison principle yields
\begin{equation}
\eta\leq z\leq\Vert z_{0}\Vert_{L^{\infty}(I)}+\eta  \leq 2\Vert z_{0}\Vert_{L^{\infty}(I)},\quad\text{in }%
I\times(0,\infty).\label{plap5}%
\end{equation}

For any $0<\tau<T<\infty$ , let us consider a test function $\xi
(t)\in\mathcal{C}_{c}^{\infty}(0,\infty)$, $0\leq\xi(t)\leq1$ such that
\[
\xi(t)=\left\{
\begin{array}
[c]{lr}%
1, & \text{on }[\tau,T],\\
& \\
0, & \text{outside }(\frac{\tau}{2},T+\frac{\tau}{2}).
\end{array}
\right.  ,\quad\text{and }|\xi_{t}|\leq\frac{c_{0}}{\tau},
\]
and put
\[
z=\varphi(v)=v^{\gamma},\quad w(x,t)=\xi(t)v_{x}^{2}.
\]
Then, we have
\begin{equation}\label{plap6}
w_{t}-aw_{xx}=\xi_{t}.v_{x}^{2}+2\xi v_{x}(v_{t}-av_{xx})_{x}-2\xi av_{xx}%
^{2}+2\xi a_{x}v_{xx}.
\end{equation}
From the equation satisfied by $z$, we get
\[
v_{t}-av_{xx}=a_{x}v_{x}+av_{x}^{2}\frac{\varphi^{\prime\prime}}
{\varphi^{\prime}}-\frac{g_\varepsilon(\varphi)}{\varphi^{\prime}}  -\frac{f(\varphi)\psi_\varepsilon (\varphi)}    {\varphi^{\prime}},
\]
 Combining the last two equations provides us
\[
w_{t}-aw_{xx}=\xi_{t}v_{x}^{2}+2\xi v_{x}\left(  a_{x}v_{x}+av_{x}^{2}
\frac{\varphi^{\prime\prime}}{\varphi^{\prime}}-\frac{g_\varepsilon(\varphi)}
{\varphi^{\prime}}  -\frac{f(\varphi)\psi_\varepsilon (\varphi)}{\varphi^{\prime}}  \right)_{x}-2\xi av_{xx}^{2}+2\xi a_{x}v_{xx}.
\]

Now, we define
\[
L=\max_{\overline{I}\times\lbrack0,\infty)}\{w(x,t)\}.
\]
If $L=0$, then the conclusion $(\ref{plapgradient})$ is trivial, and
$|z_{x}(x,\tau)|=0,\quad$in $I\times(0,\infty)$. If not we have  $L>0$, then the function $w$ must attain
its maximum at a point $(x_{0},t_{0})\in I\times(0,\infty)$ since $w(x,t)=0$
on $\partial I\times(0,\infty)$ and $w(x,t)|_{t=0}=0$. These facts lead to
\[
\left\{
\begin{array}
[c]{lr}%
w_{t}(x_{0},t_{0})=w_{x}(x_{0},t_{0})=0, & \\
\text{and} & \\
w_{xx}(x_{0},t_{0})\leq 0, &
\end{array}
\right.
\]
and  $v_{x}(x_{0},t_{0})\not =0$, so we get
\begin{equation}
w_{x}(x_{0},t_{0})=0\text{ if and only if }v_{xx}(x_{0},t_{0}%
)=0.\label{plap7b}%
\end{equation}
At the point $(x_{0},t_{0})$, $(\ref{plap6})$ and $(\ref{plap7b})$ provide us
\[
0\leq w_{t}-aw_{xx}=\xi_{t}v_{x}^{2}+2\xi v_{x}\left(  a_{xx}v_{x}+a_{x}%
v_{x}^{2}\frac{\varphi^{\prime\prime}}{\varphi^{\prime}}+av_{x}^{2}\left(
\frac{\varphi^{\prime\prime}}{\varphi^{\prime}}\right)  _{x}- \left(
\frac{g_\varepsilon(\varphi)}{\varphi^{\prime}} \right) _{x}  -\left( \frac{f(\varphi)\psi_\varepsilon (\varphi)}{\varphi^{\prime}} \right)_{x} \right)  .
\]
\begin{equation}\label{plap7}
0\leq\xi_{t}\xi^{-1}v_{x}^{2}+2v_{x}\left(  a_{xx}v_{x}+a_{x}v_{x}^{2}%
\frac{\varphi^{\prime\prime}}{\varphi^{\prime}}+av_{x}^{2}\left(
\frac{\varphi^{\prime\prime}}{\varphi^{\prime}}\right)  _{x}- \left(
\frac{g_\varepsilon(\varphi)}{\varphi^{\prime}} \right) _{x}  -\left( \frac{f(\varphi)\psi_\varepsilon (\varphi)}{\varphi^{\prime}} \right)_{x} \right) .
\end{equation}
By using again $(\ref{plap7b})$, we obtain
\begin{equation}
a_{x}(z_x)(x_{0},t_{0})=(p-2)b^{\frac{p-4}{2}}(z_x)   \varphi^{\prime}
\varphi^{\prime\prime}v^3_x,\label{plap8}%
\end{equation}
and
\begin{equation}
a_{xx}(z_x)(x_{0},t_{0})=(p-2)(p-4)b^{\frac{p-6}{2}}(z_x)  (\varphi^{\prime}%
.\varphi^{\prime\prime})^2 v_{x}^{6}+(p-2)b^{\frac{p-4}{2}}(z_x)(\varphi^{\prime
\prime2}+\varphi^{\prime}\varphi^{\prime\prime\prime})v_{x}^{4}.\label{plap9}%
\end{equation}
Next, we have
\begin{equation}
\left(  \frac{\varphi^{\prime\prime}}{\varphi^{\prime}}\right)  _{x}=\left(
\frac{\varphi^{\prime\prime\prime}\varphi^{\prime}-\varphi^{\prime\prime2}%
}{\varphi^{\prime2}}\right)  v_{x}=-(\gamma-1)v^{-2}v_{x},\label{plap9a}%
\end{equation}
and
\begin{equation*}
\left\{
\begin{array}
[c]{lr}%
v_{x}\left(  \frac{g_\varepsilon(\varphi)}{\varphi^{\prime}}\right)  _{x}=(g_{\varepsilon
}^{\prime}-g_{\varepsilon}\frac{\varphi^{\prime\prime}}{\varphi^{\prime2}%
})v_{x}^{2}=\left(  \psi_{\varepsilon}^{\prime}(\varphi)v^{-\beta}%
-(\beta+\frac{\gamma-1}{\gamma})\psi_{\varepsilon}(\varphi)v^{-(1+\beta
)\gamma}\right)  v_{x}^{2},
\\
\\
v_{x} \left( \frac{f(\varphi)\psi_\varepsilon (\varphi)}{\varphi^{\prime}} \right)_{x}  = \left((f\psi_\varepsilon)^\prime
-   (f\psi_\varepsilon)\frac{\varphi^{\prime\prime}}{\varphi^{\prime 2}}\right) v^2_x
= (f\psi_\varepsilon)^\prime v^2_x - f(\varphi(v)) .\psi_\varepsilon (\varphi(v)). (\frac{\gamma-1}{\gamma} )v^{-\gamma} v^2_x.
\end{array}
\right.  
\end{equation*}
Since $f, \psi_\varepsilon, \psi^\prime_\varepsilon\geq 0$, and $0\leq \psi_\varepsilon\leq 1$, we get 
\begin{equation}\label{plap9b}
\left\{
\begin{array}
[c]{lr}%
v_{x}\left(  \frac{g(\varphi)}{\varphi^{\prime}}\right)  _{x}\geq-(\beta
+\frac{\gamma-1}{\gamma})v^{-(1+\beta)\gamma}v_{x}^{2},
\\
\\
v_{x} \left( \frac{f(\varphi)\psi_\varepsilon (\varphi)}{\varphi^{\prime}} \right)_{x}   \geq  f^\prime(\varphi(v)) \psi_\varepsilon v^2_x
- (\frac{\gamma-1}{\gamma} ) f(\varphi(v))   v^{-\gamma} v^2_x.
\end{array}
\right.  
\end{equation}
Inserting $(\ref{plap8})$, $(\ref{plap9})$, $(\ref{plap9a})$, and
$(\ref{plap9b})$,   into $(\ref{plap7})$ yields
\[
\frac{1}{2}\xi_{t}\xi^{-1}v_{x}^{2}+  \underbrace{(p-2)(p-4)b^{\frac{p-6}{2}}(\varphi
^{\prime}\varphi^{\prime\prime})^2 v_{x}^{8}+(p-2)b^{\frac{p-4}{2}}%
(2\varphi^{\prime\prime2}+\varphi^{\prime}\varphi^{\prime\prime\prime}%
)v_{x}^{6}}_{\mathcal{B}} +
\]
\begin{equation}
(\beta+\frac{\gamma-1}{\gamma})v^{-(1+\beta)\gamma}v_{x}^{2} +  
(\frac{\gamma-1}{\gamma} )f(\varphi(v))v^{-\gamma} v^2_x- f^\prime(\varphi(v)) \psi_\varepsilon v^2_x
\geq (\gamma-1)v^{-2} a(z_x)   v_{x}^{4}.\label{plap10}%
\end{equation}
Next, we make a computation to handle $\mathcal{B}$%
\[
\begin{array}
[c]{c}%
\mathcal{B}=
(p-2)b^{\frac{p-6}{2}} (z_x)  v_{x}^{6}\left(  (p-4)(\varphi^{\prime}\varphi
^{\prime\prime})^2 v_{x}^{2}+(2\varphi^{\prime\prime2}+\varphi^{\prime}%
\varphi^{\prime\prime\prime})b(z_x)  \right)  =
\\
(p-2)\varphi^{\prime2}b^{\frac{p-6}{2}}(z_x)   v_{x}^{8}\left(  (p-2)\varphi
^{\prime\prime2}+\varphi^{\prime}\varphi^{\prime\prime\prime}\right)
+\eta^{\alpha}(p-2)(2\varphi^{\prime\prime2}+\varphi^{\prime}\varphi
^{\prime\prime\prime}) b^\frac{p-6}{2}(z_x) v_{x}^{6}=
\\
\underbrace{(p-2)(p(\gamma-1)-\gamma)\gamma^{2}(\gamma-1)v^{2(\gamma-2)}\varphi^{\prime
2} b^{\frac{p-6}{2}}(z_x)   v_{x}^{8}}_{\mathcal{B}
_{1}}  +  \underbrace{\eta^{\alpha}(p-2)\gamma^{2}(\gamma
-1)(3\gamma-4)v^{2(\gamma-2)}b^{\frac{p-6}{2}}(z_x)  v_{x}^{6}}_{\mathcal{B}_{2}}
\end{array}
\]
We observe that $\mathcal{B}_{1}\leq0$ since $p(\gamma-1)-\gamma<0$, so we
have
\begin{equation}\label{plap11}
\mathcal{B}\leq\mathcal{B}_{2}.
\end{equation}
By $(\ref{plap10})$ and $(\ref{plap11})$, we get
\[
\frac{1}{2}\xi_{t}\xi^{-1}v_{x}^{2}+(\beta+\frac{\gamma-1}{\gamma
})v^{-(1+\beta)\gamma}v_{x}^{2}+   (\frac{\gamma-1}{\gamma} )f(\varphi(v)) v^{-\gamma} v^2_x -f^\prime(\varphi(v)) \psi_\varepsilon v^2_x  +  \mathcal{B}_{2}   \geq(\gamma-1)v^{-2}a(z_x)   v_{x}^{4}.
\]
The fact that $b^{\frac{p-2}{2}}(.)$ is an increasing function since $p>2$ 
leads to
\[
a(z_x) = b^{\frac{p-2}{2}}(z_{x}) \geq  (v_{x}^{2}\varphi^{\prime2})^{\frac{p-2}{2}%
}=|v_{x}|^{p-2}\gamma^{p-2}v^{(\gamma-1)(p-2)}.
\]
From the two last inequalities, we obtain
\begin{align*}
&\frac{1}{2}\xi_{t}\xi^{-1}v_{x}^{2}+(\beta+\frac{\gamma-1}{\gamma
})v^{-(1+\beta)\gamma}v_{x}^{2}+   (\frac{\gamma-1}{\gamma} )f(\varphi(v)) v^{-\gamma} v^2_x   -
\\
& f^\prime(\varphi(v)) \psi_\varepsilon v^2_x + \mathcal{B}_{2}\geq(\gamma-1)\gamma
^{p-2}v^{(\gamma-1)(p-2)-2}|v_{x}|^{p+2}.
\end{align*}
By noting that \hspace{0.05in} $2-(\gamma-1)(p-2)=(1+\beta)\gamma$, we get
\begin{align*}
& \frac{1}{2}\xi_{t}\xi^{-1}v_{x}^{2}+(\beta+\frac{\gamma-1}{\gamma
})v^{-(1+\beta)\gamma}v_{x}^{2}+  (\frac{\gamma-1}{\gamma} )f(\varphi(v)) v^{-\gamma} v^2_x   -  \\
& f^\prime(\varphi(v)) \psi_\varepsilon v^2_x+  \mathcal{B}_{2}\geq(\gamma-1)\gamma
^{p-2}v^{-(1+\beta)\gamma}|v_{x}|^{p+2}.
\end{align*}
Multiplying both sides of the above inequality  by $v^{(1+\beta)\gamma}%
$ yields

\[
\frac{1}{2}\xi_{t}\xi^{-1}v^{(1+\beta)\gamma}v_{x}^{2}+(\beta+\frac{\gamma
-1}{\gamma})v_{x}^{2}+   (\frac{\gamma-1}{\gamma} )f(\varphi(v)) v^{\beta\gamma} v^2_x
- 
\]
\begin{equation} \label{plap14}
f^\prime(\varphi(v)) \psi_\varepsilon  v^{(1+\beta)\gamma} v^2_x + v^{(1+\beta)\gamma}\mathcal{B}_{2}     \geq(\gamma-1)\gamma^{p-2}|v_{x}%
|^{p+2}.
\end{equation}
Now, we  divide the study of inequality $(\ref{plap14})$ in two 
cases: \vspace{0.1in} \newline\textbf{(i) Case:} 
$3\gamma-4\leq0$.
\\

 We have
$\mathcal{B}_{2}\leq 0$. It
follows then from $(\ref{plap14})$ that
\begin{equation}\label{plap14a}
(\gamma-1)\gamma^{p-2}|v_{x}|^{p+2}\leq \left( \frac{1}{2}\xi_{t}\xi^{-1}
v^{(1+\beta)\gamma} + (\beta+\frac{\gamma-1}{\gamma}) + (\frac{\gamma-1}{\gamma}) f(\varphi(v)) v^{\beta\gamma}  - f^\prime(\varphi(v)) \psi_\varepsilon  v^{(1+\beta)\gamma}\right)  v^2_x .
\end{equation}
Remind that  $z=\varphi(v)=v^\gamma$. Thus,  we infer from $(\ref{plap5})$ and $(\ref{plap14a})$
\begin{equation}\label{plap14b}
|v_{x}(x_{0},t_{0})|^{2}\leq C_{1}\left(  |\xi_{t}|\xi^{-1}(t_{0})   \Vert z_{0}\Vert_{\infty}^{1+\beta}  +\Vert z_{0}\Vert_{\infty}^{\beta} . M^p_f(z_{0}) +
 \|z_{0}\|_{\infty}^{1+\beta}. M^p_{f^\prime}(z_{0}) + 1\right)  ^{\frac{2}{p}},
\end{equation}
where $C_{1}=C_{1}(\beta,p)>0$. Using Young's inequality in the right hand side of $(\ref{plap14b})$  deduces
\[
|v_{x}(x_{0},t_{0})|^{2}\leq C_{2}  \left(  |\xi_{t}(t_0)|^{\frac{2}{p}} \xi^{-\frac{2}{p}}(t_{0})   \Vert z_{0}\Vert_{\infty}^{\frac{2(1+\beta)}{p}}  +  \Vert z_{0}\Vert_{\infty}^{\frac{2\beta}{p}}. M^2_f(z_{0})  +
 \|z_{0}\|_{\infty}^{\frac{2(1+\beta)}{p}}. M^2_{f^\prime}(z_{0})  + 1  \right),
\]
with $C_2=C_2(\beta, p)$, which implies 
\begin{align*}
& w(x_{0},t_{0})=\xi(t_{0})|v_{x}(x_{0},t_{0})|^{2}  \leq 
\\
 & C_{2}. \xi(t_0) \left(  |\xi_{t}(t_0)|^{\frac{2}{p}} \xi^{-\frac{2}{p}}(t_{0})   \Vert z_{0}\Vert_{\infty}^{\frac{2(1+\beta)}{p}}  + 
\Vert z_{0}\Vert_{\infty}^{\frac{2\beta}{p}} . M^2_f(z_{0}) +
 \|z_{0}\|_{\infty}^{\frac{2(1+\beta)}{p}}. M^2_{f^\prime}(z_{0})  + 1  \right) =
\\
& C_{2} \left(  |\xi_{t}(t_0)|^{\frac{2}{p}} \xi^{1-\frac{2}{p}}(t_{0})   \Vert z_{0}\Vert_{\infty}^{\frac{2(1+\beta)}{p}}  +\xi(t_0).
\Vert z_{0}\Vert_{\infty}^{\frac{2\beta}{p}} .  M^2_f(z_{0}) +
\xi(t_0). \|z_{0}\|_{\infty}^{\frac{2(1+\beta)}{p}}.  M^2_{f^\prime}(z_{0})  + \xi(t_0) \right).
\end{align*}
Since $\xi(t)\leq1$, $|\xi_{t}(t)|\leq\displaystyle\frac{c_{0}}{\tau}$ and $w(x_{0}%
,t_{0})=\displaystyle\max_{(x,t)\in I\times\lbrack0,\infty)}\{ w(x,t)\}$, the last estimate
induces
\[
w(x,t)\leq C_{2}  \left(\tau^{-\frac{2}{p}} 
\Vert z_{0}\Vert_{\infty}^{\frac{2(1+\beta)}{p}}+  \Vert z_{0}\Vert_{\infty}^{\frac{2\beta}{p}}. M^2_f(z_{0})  +
 \|z_{0}\|_{\infty}^{\frac{2(1+\beta)}{p}}. M^2_{f^\prime}(z_{0})  + 1\right).
\]
Thus, at time $t=\tau$ we have  
\[
w(x,\tau)= \xi(\tau).|v_{x}(x,\tau)|^{2}  \displaystyle\overset{\xi(\tau)=1}{=}  |v_{x}(x,\tau)|^{2}.
\] 
Then it follows from the last inequality
\[
|v_{x}(x,\tau)|^{2}\leq C_{2}  \left(\tau^{-\frac{2}{p}} 
\Vert z_{0}\Vert_{\infty}^{\frac{2(1+\beta)}{p}}+ \Vert z_{0}\Vert_{\infty}^{\frac{2\beta}{p}} .  M^2_f(z_{0}) +
 \|z_{0}\|_{\infty}^{\frac{2(1+\beta)}{p}} . M^2_{f^\prime}(z_{0}) + 1\right),
\]
which implies
\[
 |z_{x}(x,\tau)|\leq C_{3} . z^{1-\frac{1}{\gamma}} \left(\tau^{-\frac{2}{p}} 
\Vert z_{0}\Vert_{\infty}^{\frac{2(1+\beta)}{p}}+  
\Vert z_{0}\Vert_{\infty}^{\frac{2\beta}{p}}. M^2_f(z_{0}) +
 \|z_{0}\|_{\infty}^{\frac{2(1+\beta)}{p}}. M^2_{f^\prime}(z_{0}) + 1\right)^\frac{1}{2}.
\]
Or
\[|z_{x}(x,\tau)| \leq C_{3} . z^{1-\frac{1}{\gamma}}  \left(\tau^{-\frac{1}{p}} 
\Vert z_{0}\Vert_{L^{\infty}(I)}^{\frac{(1+\beta)}{p}}+ 
\Vert z_{0}\Vert_{\infty}^{\frac{\beta}{p}}.  M_f(z_{0}) +
 \|z_{0}\|_{\infty}^{\frac{(1+\beta)}{p}} . M_{f^\prime}(z_{0}) + 1\right).
\]
This inequality holds for any $\tau>0$, so we get conclusion
$(\ref{plapgradient})$. \vspace{0.1in} \newline\textbf{(ii) Case:}
$3\gamma-4>0\Longleftrightarrow p<4(1-\beta)$.
\\

 Now $b^{\frac{p-6}{2}}(.)$ is
a decreasing function, so we have
\[
b^{\frac{p-6}{2}}(z_{x})\leq|z_{x}|^{\frac{p-6}{2}}=(v_{x}^{2}\varphi
^{\prime2})^{\frac{p-6}{2}},
\]
which implies 
\begin{align*}
 v^{(1+\beta)\gamma} \mathcal{B}_2  \leq  \eta^{\alpha}(p-2)\gamma^{2}(\gamma-1)(3\gamma-4)\gamma^{p-6}
v^{2(\gamma-2)+(1+\beta)\gamma+(\gamma-1)(p-6)}|v_x|^{p}.
\end{align*}
Note that \hspace{0.05in} $2(\gamma-2)+(1+\beta)\gamma+(\gamma-1)(p-6)=  -2(\gamma-1)$. Then, we obtain
\begin{equation}\label{plap14c}
 v^{(1+\beta)\gamma} \mathcal{B}_2  \leq 
 \eta^{\alpha}(p-2)\gamma^{2}(\gamma-1)(3\gamma-4)\gamma^{p-6}v^{-2(\gamma
-1)}|v_x|^{p}.
\end{equation}
A combination of  $(\ref{plap14c})$ and $(\ref{plap14})$  gives us
\begin{align*}
& \frac{1}{2}\xi_{t}\xi^{-1}v^{(1+\beta)\gamma}v_{x}^{2}+(\beta+\frac
{\gamma-1}{\gamma})v_{x}^{2}+ (\frac{\gamma-1}{\gamma} ) f(\varphi(v)) v^{\beta\gamma} v^2_x - f^\prime(\varphi(v)) \psi_\varepsilon  v^{(1+\beta)\gamma} v^2_x +
\\
& \eta^{\alpha}(p-2)\gamma^{2}(\gamma-1)(3\gamma-4)\gamma^{p-6}v^{-2(\gamma
-1)}|v_x|^{p}   \geq(\gamma-1)\gamma^{p-2}|v_{x}|^{p+2}.
\end{align*}
The fact $v=z^{\frac{1}{\gamma}}\geq\eta^{\frac{1}{\gamma}}$ implies
$v^{-2(\gamma-1)}\leq\eta^{-\frac{2(\gamma-1)}{\gamma}}$.
Therefore, we get
\begin{align*}
|v_{x}(x_0, t_0)|^{p+2}\leq C_{4}\left(  |\xi_{t}|\xi^{-1}v^{(1+\beta)\gamma}+1+  f(\varphi(v)) v^{\beta\gamma}- f^\prime(\varphi(v)) \psi_\varepsilon  v^{(1+\beta)\gamma}   \right)v_{x}^{2}(x_0, t_0)+
\\
C_{4}. \eta^{\alpha-\frac{2(\gamma-1)}{\gamma}}|v_x(x_0, t_0)|^{p},
\end{align*}
with $C_{4}=C_{4}(\beta,p)>0$. 
\\

Now, if $|v_{x}(x_{0},t_{0})|<1$, then we have \hspace{0.05in}
$\xi(t_{0})|v_{x}(x_{0},t_{0})|^{2}<1$,
likewise\hspace{0.05in} $w(x,t)\leq1$, in $I\times(0,\infty)$.
Thus, the conclusion $(\ref{plapgradient})$ follows immediately. 
If not, we have \hspace{0.05in}  $|v_x(x_0, t_0)|^{p}\leq |v_x(x_0, t_0)|^{p+2}$, thereby proves 
\begin{align*}
|v_{x}(x_0, t_0)|^{p+2}\leq C_{4}\left(  |\xi_{t}|\xi^{-1}v^{(1+\beta)\gamma}+  f(\varphi(v)) v^{\beta\gamma}- f^\prime(\varphi(v)) \psi_\varepsilon  v^{(1+\beta)\gamma} +1  \right)v_{x}^{2}(x_0, t_0)+
\\
C_{4}. \eta^{\alpha-\frac{2(\gamma-1)}{\gamma}}|v_x(x_0, t_0)|^{p+2},
\end{align*}
or
\[
\left(  1-C_{4}.\eta^{\alpha-\frac{2(\gamma-1)}{\gamma}}\right)  |v_x(x_0, t_0)|^{p+2}\leq C_{4}\left(  |\xi_{t}|\xi^{-1}v^{(1+\beta)\gamma}+  f(\varphi(v)) v^{\beta\gamma}- f^\prime(\varphi(v)) \psi_\varepsilon  v^{(1+\beta)\gamma} +1  \right)v_{x}^{2}(x_0, t_0).
\]
Since $\alpha>\frac{2(\gamma-1)}{\gamma}$ and $\eta\rightarrow 0^+$, there exists a positive constant $C_{5}=C_{5}(\beta,p)>0$ such that
\begin{equation*}
 |v_x (x_0, t_0)|^{p+2}\leq C_{5}\left(  |\xi_{t}|\xi^{-1}v^{(1+\beta)\gamma}+  f(\varphi(v)) v^{\beta\gamma}- f^\prime(\varphi(v)) \psi_\varepsilon  v^{(1+\beta)\gamma} +1  \right)v_{x}^{2}(x_0, t_0).
\end{equation*}
This inequality is just a version of $(\ref{plap14a})$. 
By the same analysis as in {\bf(i)},  we also obtain estimate
$(\ref{plapgradient})$. This puts an end to the proof of Lemma
\ref{lemgradientestimate}.
\end{proof}
\medskip

\begin{remark}\label{remLipbound}
If $f$ is only a global Lipschitz function with its Lipschitz constant $C_f$, then by  Rademacher's theorem (see   also in \cite{Juha}), estimate $(\ref{plapgradient})$ becomes
\begin{equation}\label{plapgradient2}
|\partial_{x} z_{\varepsilon, \eta}(x,\tau)| \leq
 C . z_{\varepsilon, \eta}^{1-\frac{1}{\gamma}} (x, \tau) \left(  \tau
^{-\frac{1}{p}}  \|z_{0}\|_{\infty}^{\frac{1+\beta}{p}} +    M_f(z_0) .\|z_{0}\|_{\infty}^{\frac{\beta}{p}}   + C^\frac{1}{p}_f .\|z_{0}\|_{\infty}^{\frac{1+\beta}{p}}+ 1 \right),
\end{equation}
for $(x,\tau) \in I\times(0,\infty)$.
\end{remark}
\bigskip

If $f$ in  Lemma \ref{lemgradientestimate}  is  a nondecreasing function, then we can relax the term containing  $M_{f^\prime}(.)$ in estimate $(\ref{plapgradient})$.    

\begin{lemma}\label{lemgradientestimate1}
Assume that $f\in\mathcal{C}^1(\mathbb{R})$ is a nondecreasing function. Then,  estimate $(\ref{plapgradient})$ can be relaxed as follows
\begin{equation}\label{plapgradient1}
|\partial_{x} z_{\varepsilon, \eta}(x,\tau)| \leq
 C . z_{\varepsilon, \eta}^{1-\frac{1}{\gamma}} (x, \tau) \left(  \tau
^{-\frac{1}{p}}  \|z_{0}\|_{\infty}^{\frac{1+\beta}{p}} +    M_f(z_0) .\|z_{0}\|_{\infty}^{\frac{\beta}{p}}  + 1 \right),
\end{equation}
for $(x,\tau) \in I\times(0,\infty)$. Note that $ M^p_f(z_0) = f(2\|z_0\|_\infty) $  since $f$ is nondecreasing.
\end{lemma}
\begin{proof}
The proof of this Lemma is most likely to the one of Lemma \ref{lemgradientestimate}. In fact, we just make a slight change in $(\ref{plap9b})$ in order to remove  the term involving  $f^\prime$. Recall here $(\ref{plap9b})$:
\[
v_{x} \left( \frac{f(\varphi)\psi_\varepsilon (\varphi)}{\varphi^{\prime}} \right)_{x}   \geq  f^\prime(\varphi(v)) \psi_\varepsilon v^2_x
- (\frac{\gamma-1}{\gamma} ) f(\varphi(v))   v^{-\gamma} v^2_x  .
\]
Since $f^\prime, \psi_\varepsilon\geq 0$, we obtain
\[
v_{x} \left( \frac{f(\varphi)\psi_\varepsilon (\varphi)}{\varphi^{\prime}} \right)_{x} 
\geq  - (\frac{\gamma-1}{\gamma} ) f(\varphi(v))   v^{-\gamma} v^2_x  .
\] 
After that, we just repeat the proof of Lemma $\ref{lemgradientestimate}$ without the term containing $f^\prime$. Thus, we get   estimate $(\ref{plapgradient1})$. 
\end{proof}

\begin{remark}\label{remextend}
We can also relax the assumption $f\in \mathcal {C}^1(\mathbb{R})$  in Lemma \ref{lemgradientestimate1} by considering the standard regularization of $f$, i.e, \hspace{0.05in} $f_n=  f*\varrho_n\in \mathcal{C}^1(\mathbb{R})$, where $\{\varrho_n\}_{n\geq 1}$ is the sequence of mollifier functions.  
\end{remark}
\bigskip

Next, we shall show that $z_{\varepsilon,\eta}$ is a
Lipschitz function on $I\times(\tau,\infty)$ with a Lipschitz constant $C$
being independent of $\varepsilon,\eta$.

\begin{proposition}
\label{propLipschitz} Assume $f$ as in Lemma \ref{lemgradientestimate}. Let $z_{\varepsilon,\eta}$ be the solution of equation
$(\ref{plap4})$ above. Then, for any $\tau>0$ there is a positive constant
$C(\beta, p, \tau, |I|, \|z_{0}\|_{\infty})$ such that
\begin{equation}
\label{plap18a}|z_{\varepsilon, \eta} (x,t) - z_{\varepsilon, \eta} (y,s) |
\leq C\left(  |x-y|+|t-s|^{\frac{1}{2}} \right)  , \quad\forall x,
y\in\overline{I}, \quad\forall t, s\geq  \tau.
\end{equation}
\end{proposition}

\begin{proof}
We first extend $z_{\varepsilon, \eta}$ by $\eta$ outside $I$, (still denoted as $z_{\varepsilon, \eta}$). To simplify the notation, we
denote again $z=z_{\varepsilon,\eta}$.
\\

Fix $\tau>0$. Multiplying equation $(\ref{plap4})$ by $\partial_{t}z$,
and using integration by parts yield
\begin{equation}
\int_{s}^{t}\int_{I}|\partial_{t}z|^{2}+a(z_{x})z_{x}\partial_{t}%
z_{x}+g_{\varepsilon}(z)\partial_{t}z + f(z)\psi_\varepsilon (z) \partial_{t}z  \hspace{0.05in}dxd\sigma=0, \quad \text{for }
t>s\geq\tau.
\label{plap18b}
\end{equation}
We observe that
\[
a(z_{x})z_{x}\partial_{t}z_{x}=\left(  |z_{x}|^{2}+\eta^{\alpha}\right)^{\frac{p-2}{2}}
.\frac{1}{2}\partial_{t}(|z_{x}|^{2})=\frac{1}{p}\partial_{t}(|z_{x}|^{2}+\eta^{\alpha
})^{\frac{p}{2}}.
\]
By this fact, we deduce from  equation $(\ref{plap18b})$
\[
\int_{s}^{t}\int_{I}|\partial_{t}z(x, \sigma)|^{2}dxd\sigma\leq\int_{I}\frac{1}{p}\left(
|z_{x}(x,s)|^{2}+\eta^{\alpha}\right)  ^{\frac{p}{2}}dx+\int_{I}
G_{\varepsilon}(z(x,s))dx  + \int_{I} H_\varepsilon (z(x,s)) dx,
\]
with
\begin{equation*}
\left\{
\begin{array}
[c]{lr}%
G_{\varepsilon}(r)=\displaystyle\int_{0}^{r}g_{\varepsilon}(s)ds\leq\displaystyle\int_{0}^{r}s^{-\beta
}ds=\frac{r^{1-\beta}}{1-\beta},
\\
\\
H_\varepsilon (r)= \displaystyle\int_{0}^{r}f(s)\psi_{\varepsilon}(s)ds\leq rf(r),\quad \text{since $f$ is nondecreasing, and $\psi_\varepsilon\leq 1$}. 
\end{array}
\right.  
\end{equation*}
Then, we obtain
\[
\int_{s}^{t}\int_{I}|\partial_{t}z(x,\sigma)|^{2} dxd\sigma\leq\frac{1}{p}\int_{I}\left(
|z_{x}(x,s)|^{2}+\eta^{\alpha}\right)  ^{\frac{p}{2}}dx+\frac{1}{1-\beta}%
\int_{I}z(x,s)^{1-\beta}dx+ \int_{I}z(x,s)f(z(x,s))  dx,
\]
or
\begin{align*}
\int_{s}^{t}\int_{I}|\partial_{t}z|^{2}dxd\sigma   \overset{(\ref{plap5})}{\leq} \frac{1}{p}\int_{I}\left(
\Vert z_{x}(s)\Vert_{\infty}^{2}+\eta^{\alpha}\right)  ^{\frac{p}{2}%
}dx+ \frac{1}{1-\beta}\int_{I}\left(  2\Vert z_{0}\Vert_{\infty}\right)  ^{1-\beta}dx   + 
\\
 \int_{I}  2\|z_0\|_{\infty}.  M^p_f( z_0) dx.
\end{align*}
We apply Young's inequality to  the first term in the right hand side to get
\begin{equation}
\int_{s}^{t}\int_{I}|\partial_{t}z|^{2}dxd\sigma\leq C_{6}\left(  \Vert
z_{x}(s)\Vert_{\infty}^{p}+\Vert z(s)\Vert_{\infty}^{1-\beta
}  + \|z_0\|_{\infty}.  M^p_f( z_0)\right)  +O(\eta), \label{plap18d}%
\end{equation}
with $C_{6}=C_{6}(\beta,p,|I|)$, and $\displaystyle\lim_{\eta\rightarrow0}O(\eta)=0$.
\\
 By $(\ref{plapgradient})$ (or $(\ref{plapgradient1})$), and $(\ref{plap18d})$, 
there is a constant $C_{7}(\beta,p,\tau,|I|,\Vert z_{0}\Vert_{\infty})>0$ such that
\begin{equation}
\int_{s}^{t}\int_{I}|\partial_{t}z|^{2}dxd\sigma\leq C_{7},\quad\forall
t>s\geq\tau. \label{plap18e}
\end{equation}
Estimate $(\ref{plap18e})$ means that $\Vert\partial_{t}z_{\varepsilon,\eta}\Vert_{L^{2}(I\times(s,t))}$ is
bounded by a constant, which is independent of $\varepsilon$ and $\eta$.
\\

 Next, for any $x,y\in I$ and for $t>s\geq \tau$, we set
\[
r=|x-y|+|t-s|^{\frac{1}{2}}.
\]
According to the Mean Value Theorem, there is a real number $\bar{x}\in I_r(y)$ such that
\begin{equation}\label{plap19}
|\partial_{t}z(\bar{x},\sigma)|^{2}=\frac{1}{|I_{r}(y)|}\int
_{I_{r}(y)}|\partial_{t}z(l,\sigma)|^{2}dl=  \frac{1}{2r}\int
_{I_{r}(y)\cap I}|\partial_{t}z(l,\sigma)|^{2}dl \leq  \frac{1}{2r}\int
_{ I}|\partial_{t}z(l,\sigma)|^{2}dl
\end{equation}
(Note that $\partial_t z(.,t)=0$ outside $I$).
\\
Next, we have  from Holder's inequality 
\[
|z(\bar{x},t)-z(\bar{x},s)|^{2}\leq(t-s)\int_{s}^{t}|\partial_{t}z(\bar
{x},\sigma)|^{2}d\sigma
  \overset{(\ref{plap19})}{\leq}   \frac{(t-s)}{2r}\int_{s}%
^{t}\int_{I}|\partial_{t}z(l,\sigma)|^{2}dld\sigma,
\]
or
\begin{equation}
|z(\bar{x},t)-z(\bar{x},s)|^{2}\leq  \frac{1}{2}(t-s)^{\frac{1}{2}}\int_{s}^{t}\int
_{I}|\partial_{t}z(l,\sigma)|^{2}dld\sigma. \label{plap18f}%
\end{equation}
From $(\ref{plap18e})$ and $(\ref{plap18f})$, we obtain
\begin{equation}
|z(\bar{x},t)-z(\bar{x},s)|^{2}\leq \frac{1}{2}C_{7} (t-s)^{\frac{1}{2}},\quad\forall
t>s\geq\tau. \label{plap18g}%
\end{equation}
Now, it is sufficient to show $(\ref{plap18a})$. Indeed, we have the triangular
inequality
\begin{align*}
|z(x,t)-z(y,s)|\leq   |z(x,t)-z(y,t)|+|z(y,t)-z(y,s)|\leq |z(x,t)-z(y,t)|+
\\
 |z(y,t)-z(\bar{x},t)|   + |z(\bar{x},t)-z(\bar{x},s)|+
  +     |z(\bar{x},s)-z(y,s)|,
\end{align*}
where $\bar{x}\in I_r(y)$ is above. Then, the conclusion $(\ref{plap18a})$ just
follows from $(\ref{plap18g})$, gradient estimates $(\ref{plapgradient})$, $(\ref{plapgradient1})$ and the Mean Value
Theorem. Or, we get  the proof of the above Proposition.
\end{proof}
\begin{remark}
The result of the above Proposition still holds for the case where $f$ is as in Lemma \ref{lemgradientestimate1} or Remark \ref{remLipbound}.
\end{remark}

Note that the estimates in the proof of Lemma \ref{lemgradientestimate} (resp. Lemma \ref{lemgradientestimate1})  and Proposition \ref{propLipschitz} are independent of $\eta$,  $\varepsilon$. This observation allows us to pass to the limit as $\eta\rightarrow0$ in order to get gradient estimates $(\ref{plapgradient})$ (resp. $(\ref{plapgradient2})$, $(\ref{plapgradient1})$)   for  the following problem 
\begin{equation}
(P_\varepsilon) \left\{
\begin{array}
[c]{lr}%
\partial_{t}z -  \partial_x \left(  |\partial_{x}z
|^{p-2}\partial_{x}z  \right) +g_{\varepsilon}(z)  + f(z) \psi_\varepsilon(z) =0 & \text{in}\hspace{0.05in}I\times(0,\infty),\\
z(-l,t)=z(l,t)=0 & \hspace{0.05in}%
t\in(0,\infty),\\
z(x,0)=z_{0}(x) & \text{on}\hspace{0.05in}I.
\end{array}
\right.  \label{plap17}%
\end{equation}

\begin{theorem}
\label{theepsilon} Let $0\leq z_{0}\in \mathcal{C}^\infty_c(I)$, $z_0\not=0$. Assume $f$ as in Lemma $\ref{lemgradientestimate}$.  Then, there
exists a unique bounded weak solution $z_{\varepsilon}$ of problem $(P_\varepsilon)$.
Furthermore, $z_\varepsilon$ also fulfills estimate $(\ref{plapgradient})$, 
and the regularity result $(\ref{plap18a})$. 
\end{theorem}
\begin{remark}\label{rem20}
The result of Theorem $\ref{theepsilon}$ also holds if $f$ is assumed as in Lemma \ref{lemgradientestimate1} (resp. Remark \ref{remLipbound}). Moreover,  $z_\varepsilon$ fulfills estimate $(\ref{plapgradient1})$ (resp. $(\ref{plapgradient2})$).
\end{remark}
\begin{proof}
 The existence and uniqueness of solution of
problem $(P_\varepsilon)$ is a classical result (see e.g \cite{WuJingHui}, \cite{Herrero}, and  \cite{Zhao}). Thanks to Lemma $\ref{lemgradientestimate}$ and the uniqueness result,  Theorem \ref{theepsilon} follows by passing $\eta\rightarrow 0$. 
\end{proof}

\section{Existence of a maximal solution}

\hspace{0.2in} In this section, we focus on the proof of Theorem $\ref{theplapexistence}$ (Theorem  $\ref{theplapexistence1}$ is proved similarly). Then, we   divide   the proof of Theorem $\ref{theplapexistence}$ into three steps.
In the first step, we prove the existence and uniqueness of solution
$u_{\varepsilon}$ of problem $(P_\varepsilon)$ with initial data $u_{0}\in
L^{1}(\Omega)$. Moreover,  we prove an estimate for $|\partial_x  u_\varepsilon|$ involving $u_\varepsilon^{1-1/\gamma}$ and $\|u_0\|_{L^1(I)}$, see Theorem \ref{theepsilon1} below. After that,  we will go to the limit as
$\varepsilon\rightarrow0$ in order to get $u_\varepsilon\rightarrow u$, a solution of equation $(\ref{plap1})$. Finally, the conclusion   that $u$ is a maximal solution is proved   in Proposition
\ref{promaxsolution} below. 
\\

We first have the following result:

\begin{theorem}\label{theepsilon1} 
Let  $0\leq u_{0}\in L^{1}(I),$ $u_{0}\not=0$. Assume that $f$ satisfies $(H_2)$. Then,
there exists a unique weak solution $u_{\varepsilon}$ of problem
$(P_\varepsilon)$ with initial data $u_0$. Moreover, for any $\tau>0$, there is a constant $C=C(\beta,p,|I|)>0$ such that
\begin{equation}
|\partial_{x}u_{\varepsilon}(x,t)|\leq C . u_{\varepsilon}^{1-\frac{1}{\gamma}}(x,t)\left(  \tau^{-\frac{\lambda+\beta+1}{\lambda p}%
}\Vert u_{0}\Vert_{L^{1}(I)}^{\frac{1+\beta}{\lambda}} +  \tau^{-\frac{\beta}{p}}\Vert u_{0}\Vert_{L^{1}(I)}^{\frac{\beta}{\lambda}} . m_f(\tau, u_0) +  1  \right), \label{plap20}
\end{equation}
for a.e $(x,t)\in(\tau,\infty)$, recall here   \hspace{0.05in} $m_f(t, u_0)= f^\frac{1}{p}\left(2C(p, |I|). t^{-\frac{1}{\lambda}}\|u_0\|^{\frac{p}{\lambda}}_{L^1(I)} \right)$.
\\

As a consequence of  $(\ref{plap20})$ and Proposition \ref{propLipschitz}, $u_{\varepsilon}$ is a Lipschitz function on $\overline{I}\times[t_1,t_2]$, for any $0<t_1<t_2<\infty$. Moreover, the Lipschitz constant of $u_\varepsilon$ is independent of $\varepsilon$.

\end{theorem}

\begin{proof}
\textbf{(i) Uniqueness:}
 The uniqueness result follows from the Lemma below.
\begin{lemma}\label{uniqueness} 
Let $v_1$ (resp. $v_2$) be a weak sub-solution (resp. super solution) of equation $(\ref{plap17})$. Then, we have
  \[
v_1 \leq v_2, \quad \text{in } I\times(0,\infty).  
  \]
\end{lemma}
\begin{proof}
We skip the proof of Lemma $\ref{uniqueness}$ and give its proof in the Appendix.
\end{proof}
\\

\textbf{(ii) Existence:} We regularize initial data $u_0$ by
considering a sequence,  $\{u_{0,n}\}_{n\geq1}\subset \mathcal{C}%
_{c}^{\infty}(I)$  such that
$u_{0,n}\overset{n\rightarrow\infty}{\longrightarrow}u_{0}$ in 
$L^{1}(I)$, and $\Vert u_{0,n}\Vert_{L^{1}(I)}\leq\Vert u_{0}
\Vert_{L^{1}(I)}$.
Let $u_{\varepsilon,n}$ be a unique (weak) solution of equation  $(\ref{plap17})$ with initial data $u_{0,n}$  (see
e.g \cite{Herrero}, \cite{Zhao}, and \cite{WuJingHui}).
We will  show that $u_{\varepsilon,n}$ converges to $u_{\varepsilon}$, which is
a solution of equation $(\ref{plap17})$ with initial data $u_0$. 
\\

First of all, we  observe that
 $u_{\varepsilon,n}$ is a sub-solution of the following
equation
\begin{equation}
\left\{
\begin{array}
[c]{lr}%
\partial_{t}v_{n}-\left(  |\partial_{x}v_{n}|^{p-2}\partial_{x}v_{n}\right)
_{x}=0 & \text{in}\hspace{0.05in}I\times(0,\infty),\\
v_{n}(-l,t)=v_{n}(l,t)=0 & \hspace{0.05in}\forall t\in(0,\infty),\\
v_{n}(x,0)=u_{0,n}(x) & \text{in}\hspace{0.05in} I,
\end{array}
\right.  \label{plap22}%
\end{equation}
thereby
\begin{equation}
u_{\varepsilon,n}\leq v_{n},\quad\text{in }I\times(0,\infty).\label{plap23}%
\end{equation}
Moreover, there is a positive constant $C(p, |I|)$ such that
\begin{equation}
\|v_{n}(.,t)\|_{\infty}  \leq C(p,|I|). t^{-\frac{1}{\lambda}} 
  \| v_{n}(0) \|_{L^{1}(I)}^{\frac{p}{\lambda}}\leq  C(p,|I|). t^{-\frac{1}{\lambda}} 
    \| u_0 \|_{L^{1}(I)}^{\frac{p}{\lambda}}
,\quad\forall t>0,\label{plap24}%
\end{equation}
(see, e.g  Theorem $4.3$, \cite{Dibe}), so we get from $(\ref{plap23})$ and $(\ref{plap24})$
\begin{equation}
\| u_{\varepsilon,n}(.,t)\|_{\infty}  \leq C(p,|I|) . t^{-\frac{1}{\lambda}}\Vert
u_{0}\Vert_{L^{1}(I)}^{\frac{p}{\lambda}},\quad\forall t>0.\label{plap27}%
\end{equation}

For any  $\tau>0$,  inequality $(\ref{plap27})$ means that $\|u(t)\|_{\infty}$ is bounded for $t\geq\tau$.
Then,  we can apply Theorem $\ref{theepsilon}$ to $u_{\varepsilon
,n}$ by considering $u_{\varepsilon,n}(\tau)$ as the initial data
 in order to get
\[
|\partial_{x}u_{\varepsilon,n}(x,t)|\leq C(\beta,p)u_{\varepsilon,n}%
^{1-\frac{1}{\gamma}}(x,t)\left(  (t-\tau)^{-\frac{1}{p}}\Vert
u_{\varepsilon,n}(\tau)\Vert_{\infty}^{\frac{1+\beta}{p}%
}+    
\Vert u_{\varepsilon,n}(\tau)\Vert_{\infty}^{\frac{\beta}{p}} .
f^{\frac{1}{p}}(2\Vert u_{\varepsilon,n}(\tau)\Vert_{\infty})  +1\right),
\]
for a.e $(x,t)\in I\times(\tau,\infty)$. In particular, we obtain for a.e $(x,t)\in I\times(2\tau,\infty)$
\begin{equation}
|\partial_{x}u_{\varepsilon,n}(x,t)|\leq C u_{\varepsilon,n}%
^{1-\frac{1}{\gamma}}(x,t)\left(  \tau^{-\frac{1}{p}}\Vert
u_{\varepsilon,n}(\tau)\Vert_{\infty}^{\frac{1+\beta}{p}%
}+    
\Vert u_{\varepsilon,n}(\tau)\Vert_{\infty}^{\frac{\beta}{p}} .
f^{\frac{1}{p}}\left( 2\Vert u_{\varepsilon,n}(\tau)\Vert_{\infty}  \right)  +1\right).\label{plap28}
\end{equation}
Recall   $m_f(t, u_0)= f^\frac{1}{p}\left(2C(p, |I|). t^{-\frac{1}{\lambda}}\|u_0\|^{\frac{p}{\lambda}}_{L^1(I)} \right)$. Combining $(\ref{plap27})$ and $(\ref{plap28})$ yields
\begin{equation}\label{plap29}
|\partial_{x}u_{\varepsilon,n}(x,t)|\leq C . u_{\varepsilon
,n}^{1-\frac{1}{\gamma}}(x,t)\left(  \tau^{-\frac{\lambda+\beta+1}{\lambda p}
}\Vert u_{0}\Vert_{L^{1}(I)}^{\frac{1+\beta}{\lambda}} +  \tau^{-\frac{\beta}{\lambda p}}\Vert u_{0}\Vert_{L^{1}(I)}^{\frac{\beta}{\lambda}} . m_f(\tau, u_0)    +1     \right),
\end{equation}
for a.e $(x,t)\in(2\tau,\infty)$.
In view of  $(\ref{plap29})$, 
$|\partial_{x}u_{\varepsilon,n}(x,t)|$ is bounded on $I\times[2\tau, \infty)$ by a positive
constant being independent of $\varepsilon$ and $n$.
Thanks to Proposition $\ref{propLipschitz}$,
we have
\begin{equation}
|u_{\varepsilon,n}(x,t)-u_{\varepsilon,n}(y,s)|\leq C\left(
|x-y|+|t-s|^{\frac{1}{2}}\right)  ,\quad\forall x,y\in\overline{I}%
,\hspace{0.05in}\forall t,s>2\tau.\label{plap29b}%
\end{equation}
Note that $C$ in $(\ref{plap29b})$ only depends on $\beta,p,\tau,|I|$, and $\Vert u_{0}\Vert_{L^{1}%
(I)}$ (instead of $\Vert u_{0}\Vert_{L^{\infty}(I)}$ as in Proposition
$\ref{propLipschitz}$). 
\\
 
Now, we can pass to the limit as $n\rightarrow \infty$. To avoid relabeling after any passage to
the limit, we want to keep the same label. Then, we observe that  $(\ref{plap29b})$ allows us to apply the Ascoli-Arzela Theorem to
$u_{\varepsilon,n}$, so there is a subsequence of $\{u_{\varepsilon
,n}\}_{n\geq1}$ such that for any $2\tau<t_1<t_2<\infty$
\begin{equation}\label{plap33b}
u_{\varepsilon,n}\overset{n\rightarrow\infty}{\longrightarrow}u_{\varepsilon
},\quad\text{uniformly on compact set }\overline{I}\times[t_1, t_2].
\end{equation}
It follows from  the diagonal argument 
that there is a subsequence of $\{u_{\varepsilon,n}\}_{n\geq1}$ such
that
\begin{equation*}
u_{\varepsilon,n}(x,t)\overset{n\rightarrow\infty}{\longrightarrow
}u_{\varepsilon}(x,t),\quad\text{pointwise in }\overline{I}\times
(0,\infty).
\end{equation*}
Thus, it is clear that  $u_\varepsilon$ also fulfills the a priori bound $(\ref{plap27})$ and the  Lipschitz continuity $(\ref{plap29b})$.
\\

After that, we show that for any $T\in(0,\infty)$
\begin{equation} \label{plap36}
g_\varepsilon(u_{\varepsilon, n})  \overset{n\rightarrow\infty}{\longrightarrow}  g_\varepsilon(u_\varepsilon), \quad \text{in }
L^1(I\times(0,T)).
\end{equation}
In fact,   $g_{\varepsilon}(.)$ is a  global Lipschitz function, and it is  bounded by $\varepsilon^{-\beta}$. 
Therefore, the Dominated Convergence Theorem yields the conclusion $(\ref{plap36})$.
\\

Next, we claim that for any $0<t_1<t_2<\infty$
\begin{equation}\label{plap43}
f(u_{\varepsilon,n})\psi_\varepsilon(u_{\varepsilon, n})  \overset{n\rightarrow\infty}{\longrightarrow} f(u_{\varepsilon}) \psi_\varepsilon(u_\varepsilon), \quad \text{in }
L^1(I\times(t_1, t_2)).
\end{equation}
According to   $(\ref{plap27})$ and  the fact $f\in \mathcal{C}(\mathbb{R})$,  we observe that $f(u_{\varepsilon, n}(x,t))$ is  bounded on $I\times(t_1, \infty)$ by a constant not depending on $\varepsilon, n$. By applying Dominated Convergence Theorem,  we get claim $(\ref{plap43})$. 
\\

Besides,  the contraction of $L^1$-norm gives us
\begin{equation}\label{plap37}
\|g_{\varepsilon}(u_{\varepsilon,n})\|_{L^1(I\times (0,\infty))}, \quad \|f(u_{\varepsilon, n})\psi_\varepsilon (u_{\varepsilon, n})\|_{L^1(I\times (0,\infty))}  \leq \|u_0\|_{L^1(I)}.
\end{equation}
It follows from $(\ref{plap37})$, $(\ref{plap43})$, and  $(\ref{plap36})$ that  
\begin{equation}\label{plap42}
\|g_{\varepsilon}(u_{\varepsilon})\|_{L^1(I\times(0,\infty))}, \quad \| f(u_{\varepsilon}) \psi_\varepsilon(u_\varepsilon)\|_{L^1(I\times(0,\infty))} \leq \|u_0\|_{L^1(I)}.
\end{equation}

Next, we show that  there is a subsequence of 
$\{u_{\varepsilon,n}\}_{ n\geq 1}$ such that
\begin{equation}\label{plap34a}
\partial_{x}u_{\varepsilon,n} (x,t) \overset{n\rightarrow\infty}{\longrightarrow}\partial_{x}u_{\varepsilon}(x,t), \quad \text{for a.e }   
(x,t)\in I\times(0,\infty). 
\end{equation}
In order to prove this,  we  borrow a result of L. Boccardo and F. Murat,
\cite{BoMu} and \cite{BoGa}, the so called \textit{almost everywhere convergence of the
gradients}. In fact, thanks to $(\ref{plap37})$, $(\ref{plap29})$, and  $(\ref{plap33b})$, we can imitate the proof in \cite{BoGa}, or \cite{BoMu} to get 
\[
\partial_{x}u_{\varepsilon,n} (x,t) \overset{n\rightarrow\infty}{\longrightarrow}\partial_{x}u_{\varepsilon}(x,t), \quad \text{for a.e }   
(x,t)\in I\times(t_1, t_2),
\]
up to a subsequence, for any $0< t_1<t_2 $. Then, the claim $(\ref{plap34a})$ just follows from the diagonal argument.
As a consequence, $u_\varepsilon$ also fulfills estimate $(\ref{plap29})$, and we have for any   $0<t_1<t_2$
\begin{equation}\label{plap34}
\partial_{x}u_{\varepsilon,n}\overset{n\rightarrow\infty}{\longrightarrow
}\partial_{x}u_{\varepsilon},\quad\text{in }L^{q}(I\times(t_1
, t_2)), \quad \text{for any } q\geq 1.
\end{equation}
By $(\ref{plap34})$, $(\ref{plap36})$, and $(\ref{plap43})$, we observe that $u_\varepsilon$ satisfies equation $(\ref{plap1})$ in the weak sense.
Then, it remains to show that 
\begin{equation}
u_{\varepsilon}\in\mathcal{C}([0,T];L^{1}(I)),\quad\text{for any }%
T\in(0,\infty).\label{plap41}
\end{equation}
Let us set
\begin{align*}
& T_k(u)=\left\{
\begin{array}
[c]{lr}%
u, &
\text{if }|u|\leq k,\\
k. sign(u),& \text{if }|u|> k,  
\end{array} \text{and}
\right.  
\\
& S_k(u)=\int^{u}_0  T_k(s) ds  =  \frac{1}{2}|u|^2\chi_{\{|u|\leq k\}}+  k(|u|-\frac{1}{2}k) \chi_{\{|u|> k\}}.
\end{align*}
We consider the difference between two equations satisfied by $u_{\varepsilon,n}$ and $u_{\varepsilon,m}$:
\begin{align*}
& \partial_t (u_{\varepsilon,n}-u_{\varepsilon,m})- \partial_x \left(  |\partial_{x}u_{\varepsilon,n}
|^{p-2}\partial_{x}u_{\varepsilon,n} \right) 
+ \partial_x \left(  |\partial_{x}u_{\varepsilon,m}
|^{p-2}\partial_{x}u_{\varepsilon,m} \right)
\\
& g_{\varepsilon}(u_{\varepsilon,n})  -g_{\varepsilon}(u_{\varepsilon,m})  +  f(u_{\varepsilon,n})\psi_\varepsilon(u_{\varepsilon, n}) -f(u_{\varepsilon,m})\psi_\varepsilon(u_{\varepsilon, m}) =0.
\end{align*}
Multiplying the above equation with $T_1(u_{\varepsilon, n}-u_{\varepsilon,m})$, and integrating on $I\times(0,t)$ yields
\begin{align*}
&\int_I S_1\left( u_{\varepsilon,n}-u_{\varepsilon,m}\right)(t) dx  + \int^t_{0}\int_I  \left(  |\partial_{x}u_{\varepsilon,n}|^{p-2}\partial_{x}u_{\varepsilon,n} 
-  |\partial_{x}u_{\varepsilon,m}|^{p-2}\partial_{x}u_{\varepsilon,m} \right)(\partial_{x}u_{\varepsilon,n}-\partial_{x}u_{\varepsilon,m}) dx ds
\\
&+\int^t_{0}\int_I \left( g_{\varepsilon}(u_{\varepsilon,n})  -g_{\varepsilon}(u_{\varepsilon,m}) \right)T_1(u_{\varepsilon,n}-u_{\varepsilon,m}) dxds + 
\\
&\int^t_{0}\int_I \left(f(u_{\varepsilon,n})\psi_\varepsilon(u_{\varepsilon, n}) -f(u_{\varepsilon,m})\psi_\varepsilon(u_{\varepsilon, m})  \right) T_1(u_{\varepsilon,n}-u_{\varepsilon,m}) dxds =
\int_I S_1\left( u_{\varepsilon,n}-u_{\varepsilon,m}\right)(0) dx.
\end{align*}
By the monotone of $p$-Laplacian operator and the monotone of the function $f \psi_\varepsilon$, we have
\begin{align*}
& \int_I S_1\left( u_{\varepsilon,n}-u_{\varepsilon,m}\right)(t) dx  \leq \int_I | u_{0,n} -u_{0,m}| dx  +
 \int^t_{0}\int_I | g_{\varepsilon}(u_{\varepsilon,n})  -g_{\varepsilon}(u_{\varepsilon,m}) | dxds  \overset  {(\ref{plap36})}{=}  o(n,m),
\end{align*}
where \hspace{0.05in} $\displaystyle \lim_{n,m\rightarrow\infty}  o(n,m)=0$.
Moreover, we have from the formula of $S_1(.)$ and Holder's inequality
\begin{equation*}
\int_{I} | u_{\varepsilon,n}-u_{\varepsilon,m}| (x,t)dx\leq\sqrt{2|I|\int_{I}S_1( u_{\varepsilon,n}-u_{\varepsilon,m}) (x,t)  dx}+2\int
_{I}S_1( u_{\varepsilon,n}-u_{\varepsilon,m})(x,t))dx,\quad\forall t>0.
\end{equation*}
Combining the two last inequalities yields
\begin{equation}
\int_{I} |u_{\varepsilon,n}-u_{\varepsilon,m}|(x,t)dx\leq C(|I|).\left(  \sqrt{o(n,m)}+o(n,m)\right)
,\quad\forall t>0.\label{plap46}
\end{equation}
Thus, $\{u_{\varepsilon,n}\}_{n\geq 1}$ is a Cauchy sequence in $\mathcal{C}([0,T]; L^1(I))$, or we get $(\ref{plap41})$.
This puts an end to the proof of Theorem \ref{theepsilon1}.
\end{proof}
\\

In the second step,  we will  pass
to the limit as $\varepsilon\rightarrow0$. Let us first claim that  $\{u_{\varepsilon}\}_{\varepsilon>0}$ is a non-decreasing sequence, so there is a nonnegative function $u$ such that $u_{\varepsilon}(x,t)\downarrow u(x,t)$ as $\varepsilon\rightarrow 0$.
Indeed, for any $\varepsilon>\varepsilon^{\prime}>0$, it is clear that
$g_{\varepsilon^{\prime}}(s)  \geq g_{\varepsilon}(s)$, and  $\psi_{\varepsilon^{\prime}}(s)  \geq \psi_{\varepsilon}(s)$\quad for $s\in\mathbb{R}$.
Therefore, $u_{\varepsilon}$ is a super-solution of equation satisfied
by $u_{\varepsilon^{\prime}}$, so Lemma \ref{uniqueness} yields
\begin{equation*}
u_{\varepsilon}(x,t)\geq u_{\varepsilon^{\prime}}(x,t),\quad\text{in }%
I\times(0,\infty), 
\end{equation*}
or  the claim follows. We would like to emphasize  that  the monotonicity of
$\{u_{\varepsilon}\}_{\varepsilon>0}$ will  be intensively used  in what
follows, although one can utilize Ascoli-Azela Theorem to show that $u_\varepsilon\rightarrow u$. Note that $u$ is also a Lipschitz function on $\overline{I}\times [t_1, t_2]$, for any $0<t_1<t_2$.
\\

To be similar to $(\ref{plap34a})$, we obtain \hspace{0.05in}
$\partial_x u_{\varepsilon}\overset{\varepsilon\rightarrow0}{\longrightarrow
}\partial_{x}u$,  for a.e $(x,t)\in I\times(0,\infty)$. As a result, $u_x$  fulfills estimate $(\ref{plap29})$, and 
\begin{equation} \label{plap46a}
\partial_{x}u_{\varepsilon}\overset{\varepsilon\rightarrow0}{\longrightarrow
}\partial_{x}u,\quad\text{in }L^{q}(I\times(t_1, t_2)),\quad \forall q\geq1,
\end{equation}
 
Next, let us show that 
\begin{equation}
u^{-\beta} \chi_{\{u>0\}}\in L^{1}(I\times(0,\infty)). \label{plap50}%
\end{equation}
 From $(\ref{plap43})$, 
applying Fatou's Lemma deduces that there is a function $\Phi\in L^{1}%
(I\times(0,\infty))$ such that
\begin{equation}
\liminf_{\varepsilon\rightarrow0}g_{\varepsilon}(u_{\varepsilon})=\Phi
,\quad\text{in }L^{1}(I\times(0,\infty)). \label{plap48}%
\end{equation}
The monotonicity of $\{u_{\varepsilon}\}_{\varepsilon>0}$ ensures  \hspace{0.05in} $g_{\varepsilon}(u_{\varepsilon})(x,t)\geq g_{\varepsilon}(u_{\varepsilon}
)\chi_{\{u>0\}}(x,t)$, so
\begin{equation}
\liminf_{\varepsilon\rightarrow0}g_{\varepsilon}(u_{\varepsilon})(x,t)=\Phi\geq
u^{-\beta}\chi_{\{u>0\}}(x,t),\quad\text{ for a.e }(x,t)\in I\times(0,\infty).
\label{plap49}%
\end{equation}
Thus,  conclusion $(\ref{plap50})$ just follows from $(\ref{plap48})$ and $(\ref{plap49})$.
Actually, we will show at the end of this step  that 
\begin{equation}
\liminf_{\varepsilon\rightarrow 0} g_{\varepsilon}(u_{\varepsilon})
= \Phi=u^{-\beta}\chi_{\{u>0\}},\quad\text{in }L^{1}(I\times
(0,\infty)). \label{plap47}%
\end{equation}
Let us emphasize that $(\ref{plap47})$ implies the conclusion
\begin{equation}\label{plap47a}
u\in\mathcal{C}([0,\infty);L^{1}(I)),
\end{equation}
by following the proof of $(\ref{plap41})$. 
\\

At the moment, we demonstrate that $u$ must satisfy equation $(\ref{plap1})$ in the sense of distribution.
\\
For any $\eta>0$ fixed, we use the test function $\psi_{\eta
}(u_{\varepsilon})\phi$, $\phi\in\mathcal{C}_{c}^{\infty}(I\times(0,\infty))$, in
the equation satisfied by $u_{\varepsilon}$. Then, using integration by parts yields
\begin{align*}
& \int_{Supp(\phi)}-\Psi_{\eta}(u_{\varepsilon})\phi_{t}+\frac{1}{\eta}%
|\partial_x u_{\varepsilon}|^{p}\psi{^{\prime}}(\frac{u_{\varepsilon}}{\eta
})\phi+   |\partial_{x}u_{\varepsilon}|^{p-2}\partial_{x}u_{\varepsilon}%
. \phi_x  . \psi_{\eta}(u_{\varepsilon})+\\
& g_{\varepsilon}(u_{\varepsilon
})\psi_{\eta}(u_{\varepsilon})\phi + f(u_\varepsilon)\psi_\eta(u_\varepsilon). \psi_\varepsilon(u_\varepsilon) \phi  \hspace{0.05in}dxds=0,
\end{align*}
with  \hspace{0.05in}
$\displaystyle\Psi_{\eta}(u)=\int_{0}^{u}\psi_{\eta}(s)ds.$
There is no problem of going to the limit as $\varepsilon\rightarrow0$ in the indicated equation, so we have
\begin{equation*}
\int_{Supp(\phi)}-\Psi_{\eta}(u)\phi_{t}+\frac{1}{\eta}|u_x|^{p}%
\psi^{\prime}(\frac{u}{\eta})\phi+| u_x|^{p-2}
u_x .\phi_x .\psi_{\eta}(u)+u^{-\beta}\psi_{\eta}(u)\phi + f(u)\psi_\eta(u) \phi \hspace
{0.05in}dxds=0. 
\end{equation*}

Next, we go to the limit as  $\eta\rightarrow0$ in the above equation. From $(\ref{plap27})$, $(\ref{plap29})$, $(\ref{plap46a})$, and $(\ref{plap50})$, it is not difficult to verify
\begin{equation}\label{plap52}
\left\{
\begin{array}
[c]{l}%
\displaystyle\lim_{\eta\rightarrow0}\int_{Supp(\phi)}\Psi_{\eta}(u)\phi_{t} \hspace
{0.05in} dxds=\int
_{Supp(\phi)}u.\phi_{t}\hspace
{0.05in}dxds,
\\
\\
\displaystyle\lim_{\eta\rightarrow0}\int_{Supp(\phi)}|u_x|^{p-2} u_x. \phi_x .\psi_{\eta}(u)\hspace{0.05in}dxds=\int_{Supp(\phi)}|u_x|^{p-2} u_x.\phi_x \hspace{0.05in} dxds,
\\
\\
\displaystyle\lim_{\eta\rightarrow0}\int_{Supp(\phi)}u^{-\beta} \psi_{\eta
}(u)\phi\hspace
{0.05in} dxds=\int_{Supp(\phi)}u^{-\beta}\chi_{\{u>0\}}\phi \hspace
{0.05in}dxds,
\\
\\
\displaystyle\lim_{\eta\rightarrow0}\int_{Supp(\phi)} f(u)\psi_\eta(u) \phi
\hspace{0.05in} dxds=\int_{Supp(\phi)} f(u)  \phi \hspace
{0.05in}dxds. 
\end{array}
\right.  
\end{equation}
Note that the assumption  $f(0)=0$ is used in the final equality of $(\ref{plap52})$.
While, we have
\begin{equation}
\lim_{\eta\rightarrow0}\int_{Supp(\phi)}\frac{1}{\eta}|\partial_{x}u|^{p}%
\psi{^{\prime}}(\frac{u}{\eta})\phi dxds=0. \label{plap52a}
\end{equation}
In fact,  since $u$ satisfies estimate $(\ref{plap29})$, we have
\begin{align*}
\frac{1}{\eta}\int_{Supp(\phi)}|\partial_{x}u|^{p}|\psi{^{\prime}}(\frac
{u}{\eta}).\phi|dxds  &  \leq C\frac{1}{\eta}\int_{Supp(\phi)\cap
\{\eta<u<2\eta\}}u^{1-\beta}dxds\\
&  \leq 2C\int_{Supp(\phi)\cap\{\eta<u<2\eta\}}u^{-\beta}dxds,
\end{align*}
where the constant $C>0$ merely depends on $\beta$, $p$, $Supp(\phi)$, $\Vert
u_{0}\Vert_{L^{1}(I)}$.  Moreover,  $u^{-\beta}\chi_{\{u>0\}}$ is integrable on $I\times(0,\infty)$ by $(\ref{plap50})$. Then, we obtain
\[
\lim_{\eta\rightarrow0}\int_{Supp(\phi)\cap\{\eta<u<2\eta\}}u^{-\beta}dxds=0,
\]
which implies the conclusion $(\ref{plap52a})$. A combination of 
$(\ref{plap52})$ and $(\ref{plap52a})$ deduces
\begin{equation}
\int_{Supp(\phi)}\left(  -u\phi_{t}+|u_x|^{p-2}
u_x \phi_x + u^{-\beta}\chi_{\{u>0\}}\phi  + f(u)\phi \right)  dxds=0. \label{plap54}%
\end{equation}
In  other words, $u$ satisfies equation $(\ref{plap1})$ in $\mathcal{D^{\prime}
}(I\times(0,\infty))$. 
\\

As mentioned above, we prove $(\ref{plap47})$ now. The fact that $u_{\varepsilon}$ is a weak
solution of $(\ref{plap17})$ gives us
\[
\int_{Supp(\phi)}\left(  -u_{\varepsilon}\phi_{t}+|\partial_{x}u_{\varepsilon
}|^{p-2}\partial_{x}u_{\varepsilon}\partial_{x}\phi+g_{\varepsilon
}(u_{\varepsilon})\phi  +  f(u_\varepsilon) \psi_\varepsilon(u_\varepsilon) \phi \right)  dxds=0,
\]
for $\phi\in\mathcal{C}_{c}^{\infty
}(I\times(0,\infty))$, $\phi\geq0$.
Letting $\varepsilon\rightarrow0$ induces
\begin{equation}
\int_{Supp(\phi)}\left(  -u\phi_{t}+|u_x|^{p-2}
u_x \phi_x \right)  dxds + \lim_{\varepsilon\rightarrow0}\int_{Supp(\phi
)}g_{\varepsilon}(u_{\varepsilon})\phi \hspace{0.05in} dxds+ \int_{Supp(\phi)} f(u)\phi  \hspace{0.05in}dxds=0. \label{plap55}%
\end{equation}
By $(\ref{plap54})$ and $(\ref{plap55})$, we get
\begin{equation}
\lim_{\varepsilon\rightarrow0}\int_{0}^{\infty}\int_{I}g_{\varepsilon
}(u_{\varepsilon})\phi \hspace{0.05in}dxds=\int_{0}^{\infty}\int_{I}u^{-\beta}\chi
_{\{u>0\}}\phi \hspace{0.05in}dxds. \label{plap56}%
\end{equation}
According to $(\ref{plap48})$, $(\ref{plap56})$ and Fatou's Lemma, we obtain
\[
\int_{0}^{\infty}\int_{I}u^{-\beta}\chi_{\{u>0\}}\phi dxds\geq\int_{0}%
^{\infty}\int_{I}\Phi\phi dxds,\quad\forall\phi\in\mathcal{C}_{c}^{\infty
}(I\times(0,\infty)),\phi\geq0.
\]
The last inequality and $(\ref{plap49})$ yield
\[
u^{-\beta}\chi_{\{u>0\}}=\Phi,\quad\text{a.e in }I\times(0,\infty).
\]
Thereby, we get $(\ref{plap47})$. In conclusion, $u$ is a weak solution of equation
$(\ref{plap1})$.

\begin{remark}
The reader should note that $(\ref{plap54})$ is not sufficient to conclude
that $u$ is a weak solution of equation $(\ref{plap1})$ according to Definition $1$. Thus, it is
necessary to prove  $(\ref{plap47})$, thereby proves $(\ref{plap47a})$.
\end{remark}
\bigskip

We end this Section by proving that $u$ is the maximal solution of equation
$(\ref{plap1})$.
\begin{proposition}
\label{promaxsolution} Let $v$ be any weak solution of equation $(\ref{plap1}%
)$. Then, we have
\[
v (x,t)\leq u(x,t), \quad\text{for a.e }(x,t)\in I\times(0,\infty).
\]
\end{proposition}
\begin{proof}
 For any $\varepsilon>0$, we observe that
\[
g_{\varepsilon}(v)\leq v^{-\beta}\chi_{\{v>0\}}, \text{and } f(v).\psi_\varepsilon(v)\leq f(v).
\]
Thus, we get
\[
\partial_{t}v-\left(  |v_x|^{p-2} v_x\right)
_{x}+g_{\varepsilon}(v) + f(v).\psi_\varepsilon(v)
\leq \partial_{t}v-\left(  |v_x|^{p-2}
v_x\right)  _{x}+v^{-\beta}\chi_{\{v>0\}}+  f(v)   =0,
\]
which implies that $v$ is a sub-solution of equation $(P_\varepsilon)$. Thanks to Lemma \ref{uniqueness}, we get
\[
v(x,t)\leq u_{\varepsilon}(x,t),\quad\text{for a.e }(x,t)\in I\times
(0,\infty).
\]
Letting $\varepsilon\rightarrow0$ yields the result. This puts an end to the
proof of Theorem $\ref{theplapexistence}$.
\end{proof}
\bigskip

If $f$ is a global Lipschitz function, then we have 
\begin{theorem}\label{thegloLip} 
Let $0\leq u_0\in L^1(I)$, $u_0\not=0$.
Assume that $f$ is a global Lipschitz  function with Lipschitz constant $C_f$, and  $f(0)=0$.   Then there exists a maximal weak solution $u$ of equation $(\ref{plap1})$.  Furthermore, we have 
\\

 For any $\tau>0$, there exist two positive constants $C_1(\beta,p,|I|)$ and $C_2(p, |I|)$ 
such that  
\begin{equation}
|u_x(x,t)|\leq C_1 . u^{1-\frac{1}{\gamma}}(x,t)\left(  \tau^{-\frac{\lambda+\beta+1}{\lambda p}%
}\Vert u_{0}\Vert_{L^{1}(I)}^{\frac{1+\beta}{\lambda}} +  \tau^{-\frac{\beta}{\lambda p}}\Vert u_{0}\Vert_{L^{1}(I)}^{\frac{\beta}{\lambda}} . M_f(u(\tau))  + C^{\frac{1}{p}}_f . \tau^{-\frac{1+\beta}{\lambda p}}\|u_0\|^{\frac{1+\beta}{\lambda}}_{L^1(I)} +1     \right), \label{plap56a}
\end{equation}
for a.e $(x,t)\in I\times(\tau,\infty)$. Note that  \hspace{0.05in}
$M_f(u(\tau))  \leq  \left(\displaystyle\max_{0\leq s\leq C_2 . \tau^{-\frac{1}{\lambda}}\|u_0\|^{\frac{p}{\lambda}}_{L^1(I)}}|f(s)|  \right)^\frac{1}{p}$.
\end{theorem}
\begin{proof}
The proof of this Theorem is most likely to the one of Theorem \ref{theplapexistence}. Then, we leave it for the  reader, who is interested in detail. Note that estimate $(\ref{plap56a})$ is just a combination of the a priori bound $(\ref{plap27})$ and $(\ref{plapgradient2})$.
\end{proof}

\begin{remark}
We emphasize that our existence results also holds for a class of continuous functions $f(u,x,t): \mathbb{R}\times I\times (0,\infty)\rightarrow  [0,\infty)$, such that
$f(0, x,t)=0$, $\forall (x,t)\in I\times (0,\infty)$; and for any $(x,t)\in I\times (0,\infty)$, $f(., x,t)$ satisfies either
$(H_1)$, or   $(H_2)$,  or a global Lipschitz property.
\end{remark}

\section{Quenching phenomenon of nonnegative solutions}

\hspace{0.2in} In this section, we will show that any weak solution of equation
$(\ref{plap1})$ must quench  (Theorem \ref{theplapqueching1} and Theorem \ref{theplapqueching}).
According to Proposition \ref{promaxsolution}, it is enough to prove that the
maximal solution $u$  vanishes identically after a finite time. Then, we have the following result

\begin{theorem}
\label{theextinction} Let  $u_{0}\in L^{1}(I)$, $u_{0}\geq0$, and $f$ satisfy $(H_2)$.  Then,
there exists a finite time $T_{0}$ such that
\[
u(x,t)=0, \quad\forall x\in\overline{I},  \forall t>T_{0}.
\]
Furthermore, $T_{0}$ can be estimated by  a constant depending on $\beta, p, |I|, \|u_{0}\|_{L^{1}(I)}$.
\end{theorem}

\begin{proof}
For any $\tau>0$,  we put
\[
L\left(\tau, u_0\right)= C(p, |I|) . \tau^{-\frac{1}{\lambda}}\|u_0\|^{\frac{p}{\lambda}}_{L^1(I)},
\]
 the a priori bound of $u(x,t)$ on  $[\tau, \infty)$, see $(\ref{plap27})$.
\\

Let $\Gamma_\varepsilon(t)$ be a  solution of equation
 \begin{equation}\label{plap57a}
 \left\{
 \begin{array}
 [c]{l}%
 \partial_{t} \Gamma_{\varepsilon}(t)+g_{\varepsilon}(\Gamma_{\varepsilon})=0\quad t>0,\\
 \\
 \Gamma_{\varepsilon}(0)=L\left(\tau, u_0\right).
 \end{array}
 \right.
 \end{equation}
Then, $\Gamma_\varepsilon$ is a super-solution of equation $(P_\varepsilon)$  satisfied by $u_\varepsilon$. Therefore, a comparison deduces
\[
u_\varepsilon (x,s+\tau)\leq \Gamma_\varepsilon(s), \quad   \forall (x,s)\in I\times(0,\infty). 
\]
It is straightforward to show that
 \[
 \Gamma_{\varepsilon}(t)  \overset{\varepsilon\rightarrow  0}{\longrightarrow} \Gamma(t)=\left(  
 L\left(\tau, u_0\right)^{1+\beta}-(1+\beta)t\right)^{\frac{1}{1+\beta}}_+, \quad \text{for } t>0.
 \]
Then, we obtain
 \[
 u(x,s+\tau)\leq \Gamma(s),\quad\text{for }  (x,s)\in I\times(0,\infty),
 \]
 which implies
 \begin{equation}\label{plap57}
 u(x,t)=0, \quad \text{for any } t\geq \tau+  \frac{1}{1+\beta}L^{1+\beta}\left(\tau, u_0\right), \text{ and for }  x\in I.
 \end{equation}
 
 Now, we try to estimate the value of the minimal extinction time $T_{0}$. It follows from $(\ref{plap57})$ that
 \[
 T_0\leq  \tau+  \frac{1}{1+\beta}L^{1+\beta}\left(\tau, u_0\right), \quad \forall \tau>0,
 \]
thereby
\begin{equation*}
 T_0\leq \min_{\tau>0}\{ \tau+  \frac{1}{1+\beta}L^{1+\beta}\left(\tau, u_0\right)\}= 
 C(\beta, p, |I|) \|u_0\|^{\frac{(1+\beta)p}{1+\beta+\lambda}}_{L^1(I)}. 
\end{equation*}
This completes the proof of Theorem $\ref{theextinction}$, thereby proves Theorem \ref{theplapqueching}.
\end{proof}
\begin{remark}
The result of Theorem \ref{theextinction}
still holds if $f$ is assumed as in Theorem \ref{thegloLip}. 
\end{remark}
\begin{remark}\label{rem15} Theorem \ref{theplapqueching1} is proved similarly. Furthermore, 
 $T_0$ can be estimated by the constant $\displaystyle\frac{\|u_0\|^{1+\beta}_\infty}{1+\beta}$, see also Theorem $\ref{theextinctionCauchy}$.
\end{remark}
\bigskip

As a consequence of Theorem $\ref{theextinction}$,  the existence result  fails if $f(0)>0$.
\begin{corollary}\label{Cornonexist}
  Let $f(u, x, t): [0,\infty)\times I\times (0,\infty) \rightarrow [0,\infty)$ be a real nonnegative function. Assume that there is a point $x_0\in I$ such that  $f(0,x_0,t)>0$, for any $t>0$ large enough.  Then, we have no nonnegative weak solution of  problem $(\ref{plap1})$.
  \end{corollary} 
\begin{proof}
If the conclusion were false, there would exist then a weak solution  of $(\ref{plap1})$, say $\bar{u}$. Thus, $\bar{u}$ is a sub-solution of equation $(\ref{plap17})$.
Use the same analysis as in the proof of Theorem \ref{theextinction} to obtain
 \[
 \bar{u}(x,s+\tau)\leq \Gamma(s),\quad\text{for }  (x,s)\in I\times(0,\infty),
 \]
 which implies that $\bar{u}(x,t)$ must vanish identically after a finite time $T_0$. In particular, we have from equation satisfied by $\bar{u}$ that $f(0, x_0, t)=0$,  for any $t>T_0$. This contradicts the above assumption, or we get Corollary \ref{Cornonexist}.
\end{proof}
\begin{remark} It is of course that
our conclusion in Corollary \ref{Cornonexist} also holds for the inequalities  $u\geq 0$ and
\begin{equation*}
\left\{
\begin{array}
[c]{lr}%
\partial_{t}u-(|u_{x}|^{p-2}u_{x})_{x}+u^{-\beta}\chi_{\{u>0\}}   + f(u,x,t)\leq 0 &
\text{in}\hspace{0.05in}I\times(0,\infty),\\
u(-l,t)=u(l,t)=0 & t\in(0,\infty),\\
u(x,0)=u_{0}(x) & \hspace{0.05in}\text{in}\hspace{0.05in}I.
\end{array}
\right.  
\end{equation*}
\end{remark}

\section{On the associated Cauchy problem}
\hspace{0.2in} In this section, we  prove  the existence results for the Cauchy problem  $(\ref{plapCauchy})$.
Furthermore, the behaviors of nonnegative solutions are considered such as the quenching phenomenon, the finite speed of propagation, and the ISS property.

\subsection{Existence of a weak solution}
\hspace{0.2in} As mentioned in the Introduction, we first have an existence result of problem $(\ref{plapCauchy})$. 

\begin{theorem}
\label{theexistenceCauchy} Let $0\leq u_{0}\in
L^{1}(\mathbb{R})\cap L^{\infty}(\mathbb{R})$. Assume that $f$ satisfies either $(H_1)$, or $(H_2)$, or a global Lipschitz property. Then, there exists a weak bounded solution
$u\in\mathcal{C}([0,\infty);L^{1}(\mathbb{R}))\cap L^p(0,T; W^{1,p}(\mathbb{R}))$, satisfying
equation $(\ref{plapCauchy})$ in $\mathcal{D^{\prime}}(\mathbb{R}%
\times(0,\infty))$. Furthermore,  $u$ satisfies estimate $(\ref{plapgradient})$ corresponding to $(H_1)$ (resp. estimate $(\ref{plapgradient1})$ corresponding to $(H_2)$, and estimate $(\ref{plapgradient2})$ corresponding to the assumption global Lipschitz).
\end{theorem}

\begin{proof} We only give the proof of the case $(H_2)$. The case $(H_1)$ (resp. global Lipschitz) is proved similarly. 
\\

Let $u_r$ be the maximal solution of the following equation 
\begin{equation}
\left\{
\begin{array}
[c]{lr}%
\partial_{t}u-(|u_{x}|^{p-2}u_{x})_{x}+ u^{-\beta}\chi_{\{u>0\}} + f(u)=0 & \text{in}%
\hspace{0.05in} I_r\times(0,\infty),\\
u(-r,t)=u(r,t)=0, & \hspace{0.05in}\forall t\in(0,\infty),\\
u(x,0)=u_{0}(x), & \hspace{0.05in}\text{in}\hspace{0.05in} I_r,
\end{array}
\right.  \label{plap70}%
\end{equation}
see Theorem $\ref{theplapexistence}$.  It is clear that  $\{u_r\}_{r>0}$ is a nondecreasing sequence. Moreover, 
 the strong comparison principle deduces
\begin{equation}
 u_{r}(x,t) \leq\| u_{0}
\|_{L^{\infty}(\mathbb{R})},\quad\text{for } (x,t)\in  I_r\times(0,\infty).
\label{plap71}%
\end{equation}
Thus, there exists a function $u$ such that $u_r \uparrow u$ as $r\rightarrow\infty$. We will show that $u$ is  a solution of problem $(\ref{plapCauchy})$.
\\

First, $L^{1}$-contraction provides us
\begin{equation}
\left\{
\begin{array}
[c]{l}%
\Vert u_{r}(.,t)\Vert_{L^{1}(I_r)}\leq\Vert u_{0}\Vert
_{L^{1}(\mathbb{R})},\quad\text{for any } t\in(0,\infty), 
\\
\\
 \|f(u_r)\|_{L^{1}(I_r\times(0,\infty))}, \hspace{0.05in}\|u^{-\beta}_r\chi_{\{u_r>0\}}\|_{L^{1}(I_r\times(0,\infty))}\leq \Vert u_{0}\Vert
_{L^{1}(\mathbb{R})}.
\end{array}
\right.  \label{plap71g}
\end{equation}
It follows immediately from the Monotone Convergence Theorem  that $ u_{r}(t)$ converges to $u(t)$ in $L^{1}(\mathbb{R})$ and $f(u_r)$ converges to $f(u)$ in $L^{1}(\mathbb{R}\times(0,\infty))$ as $r\rightarrow\infty$, likewise
\begin{equation}\label{plap71f}
\left\{
\begin{array}
[c]{l}%
\Vert u(.,t)\Vert_{L^{1}(\mathbb{R})}\leq\Vert u_{0}\Vert
_{L^{1}(\mathbb{R})},\quad\text{for any }t\in(0,\infty), 
\\
\\
\|f(u)\|_{L^{1}(\mathbb{R}\times(0,\infty))}\leq\Vert u_{0}\Vert
_{L^{1}(\mathbb{R})}.
\end{array}
\right. 
\end{equation}

Next, for any $r>0$, we have from $(\ref{plapgradient})$ (see also Theorem $\ref{theepsilon}$) 
\begin{equation}\label{plap71h}
|\partial_{x} u_{r}(x, t)| \leq
 C . u_{r}^{1-\frac{1}{\gamma}} (x, t) \left(  t^{-\frac{1}{p}} 
  \|u_{0}\|_{L^{\infty}(\mathbb{R})}^{\frac{1+\beta}{p}} + \|u_{0}\|_{L^{\infty}(\mathbb{R})}^{\frac{\beta}{p}}. f^{\frac{1}{p}} \left( 2\|u_{0}\|_{L^{\infty}(\mathbb{R})}\right )   + 1 \right), 
\end{equation}
for a.e $(x,t) \in I_r\times(0,\infty)$. By using the same argument as in the proof of   $(\ref{plap34a})$,   there is a subsequence of $\{u_r\}_{r>0}$ such that
 \hspace{0.05in}
$\partial_x u_r  \overset{r\rightarrow \infty}{\longrightarrow}  u_x$,  \text{ for a.e } $(x,t)\in \mathbb{R}\times(0,\infty)$. From this result and 
 $(\ref{plap71h})$, we obtain
\begin{equation}\label{plap72a}
| u_x(x, t)| \leq
 C . u^{1-\frac{1}{\gamma}} (x, t) \left(  t^{-\frac{1}{p}} 
  \|u_{0}\|_{L^{\infty}(\mathbb{R})}^{\frac{1+\beta}{p}} + \|u_{0}\|_{L^{\infty}(\mathbb{R})}^{\frac{\beta}{p}}. f^{\frac{1}{p}} \left( 2\|u_{0}\|_{L^{\infty}(\mathbb{R})}\right )   + 1 \right),
\end{equation}
for a.e $(x,t) \in \mathbb{R}\times(0,\infty)$, and
\begin{equation}\label{plap72d}
\partial_x u_r  \overset{r\rightarrow \infty}{\longrightarrow}  u_x, \quad \text{in } L^q_{loc}(\mathbb{R}\times(0,\infty)), \quad \forall q\geq 1.
\end{equation} 

Now, we show that $u$ satisfies equation $(\ref{plapCauchy})$ in the sense of distribution. Indeed, using the test function $\psi_\eta(u_r).\phi$ for the equation satisfied by $u_r$ gives us
\begin{align*}
& \int_{Supp(\phi)} -\Psi_\eta(u_r) \phi_t  + |\partial_x u_r|^{p-2}\partial_x u_r .\phi_x \psi_\eta(u_r) + \frac{1}{\eta}|\partial_x u_r|^{p-2}\partial_x u_r  . \psi^{\prime}(\frac{u_r}{\eta})\phi
\\
& +u^{-\beta}_r\chi_{\{u_r>0\}}\psi_\eta(u_r)\phi+ f(u_r)\psi_\eta(u_r)\phi  \hspace{0.05in} dsdx
=0, \quad \forall \phi\in\mathcal{C}^\infty_{c}(\mathbb{R}\times(0,\infty)).
\end{align*}
We first take care of the term  $u^{-\beta}_r \chi_{\{u_r>0\}}\psi_\eta(u_r)\phi$ in passing $r\rightarrow \infty$ and $\eta\rightarrow 0$. It is not difficult to see that  
 $u^{-\beta}_r\chi_{\{u_r>0\}}\psi_\eta(u_r)=u^{-\beta}_r \psi_\eta(u_r)$ is bounded by $\eta^{-\beta}$. Then for any $\eta>0$, the Dominated Convergence Theorem yields    
 $u^{-\beta}_r \psi_\eta(u_r)  \displaystyle\overset{r\rightarrow \infty}{\longrightarrow } u^{-\beta} \psi_\eta(u)$ \hspace{0.05in} in  $L^1_{loc}(\mathbb{R}\times(0,\infty))$,  which  implies 
\[
\|u^{-\beta} \psi_\eta(u)\|_{L^1(\mathbb{R}\times(0,\infty))}  \overset{(\ref{plap71g})}{\leq} \|u_0\|_{L^1(\mathbb{R})}.
\]
Next, using the Monotone Convergence Theorem deduces $ u^{-\beta} \psi_\eta(u)\uparrow u^{-\beta}\chi_{\{u>0\}}$ as $ \eta\rightarrow 0$, thereby proves
\begin{equation}\label{plap73}
 \|u^{-\beta}\chi_{\{u>0\}}\|_{L^1(\mathbb{R}\times(0,\infty))}  \leq \|u_0\|_{L^1(\mathbb{R})}.
\end{equation}
Thanks to $(\ref{plap72d})$, $(\ref{plap71g})$ and $(\ref{plap71})$,  there is no problem of  passing to the limit as $r\rightarrow \infty$ in the indicated variational equation  in order to get 

\begin{align*}
& \int_{Supp(\phi)}- \Psi_\eta(u) \phi_t  + |u_x|^{p-2}  u_x .\phi_x \psi_\eta(u) + \frac{1}{\eta}| u_x|^{p-2} u_x  . \psi^{\prime}(\frac{u}{\eta})\phi
\\
& +u^{-\beta} \psi_\eta(u) \phi+ f(u)\psi_\eta(u)\phi  \hspace{0.05in} dsdx
=0, \quad \forall \phi\in\mathcal{C}^\infty_{c}(\mathbb{R}\times(0,\infty)).
\end{align*}
By $(\ref{plap71f})$,  $(\ref{plap72a})$, and $(\ref{plap73})$,  we can make the same argument as in $(\ref{plap52})$ and  $(\ref{plap52a})$ to obtain after   letting $\eta\rightarrow0$
\begin{equation}\label{plap73c}
 \int_{Supp(\phi)} -u \phi_t  + |u_x|^{p-2}  u_x .\phi_x  + 
u^{-\beta} \chi_{\{u>0\}} \phi+ f(u)\phi  \hspace{0.05in} dxds
=0, \quad \forall \phi\in\mathcal{C}^\infty_{c}(\mathbb{R}\times(0,\infty)).
\end{equation}
Or $u$ satisfies  equation $(\ref{plap1})$ in the sense of distribution. 
\\

Then, it remains to prove that $u\in \mathcal{C}([0,\infty); L^1(\mathbb{R}))$. 
Let us first claim that 
\begin{equation}\label{plap73b}
u\in \mathcal{C}([0,\infty); L^1_{loc}(\mathbb{R})).
\end{equation}
In order to prove $(\ref{plap73b})$, we borrow a compactness result of A. Porretta \cite{Porretta}. We present it here for a convenience
\begin{lemma}[Theorem $1.1$, \cite{Porretta}]\label{Lemcompactness}
Let $p>1$ and $p'$ its conjugate exponent $\left(\frac{1}{p}+ \frac{1}{p'}=1 \right)$,
$a, b\in\mathbb{R}$, and define the space
\begin{align*}
V^p_1(a,b)  =  \{u:  \Omega\times(a,b) \rightarrow \mathbb{R}; \quad u\in L^p(a,b; W^{1,p}_0(\Omega) ),\\
u_t \in L^{p'}(a,b; W^{-1, p'}(\Omega))+ L^1(\Omega\times(a,b))   \},
\end{align*} 
where  $\Omega$ is a bounded set in  $\mathbb{R}^N$. Then, we have 
\[
V^p_1(a,b)   \subset \mathcal{C}([a,b]; L^1(\Omega)).
\]
\end{lemma}
\begin{proof}
See its proof in Theorem $1.1$, \cite{Porretta}.
\end{proof}
\\

For any $r>0$, we extend $u_r$ by $0$ outside $I_r$, still denoted as $u_r$.  Use $u_r$ as a test function to the equation satisfied by $u_r$ to  get 
\[
\int^T_0\int_{\mathbb{R}} |\partial_x u_r|^p dxds \leq \frac{1}{2}\int_{I_r}u^2_0(x)dx\leq \frac{1}{2}\|u_0\|_{L^1(\mathbb{R})} \|u_0\|_{L^\infty(\mathbb{R})}, \quad \text{for } T>0 .
\]
Thus $\displaystyle \|u_x\|^p_{L^p(\mathbb{R}\times(0,T))}$ is also bounded by $\frac{1}{2}\|u_0\|_{L^1(\mathbb{R})} \|u_0\|_{L^\infty(\mathbb{R})}$. By  $(\ref{plap71f})$ and the boundedness of $u$, it follows from the Interpolation Theorem  that $u\in L^p(\mathbb{R}\times(0,T))$, for any $T>0$. Thus, we have \hspace{0.05in}
$u\in L^p(0,T ; W^{1,p}(\mathbb{R}) )$.
 \\
According to this conclusion, $(\ref{plap71f})$ and  $(\ref{plap73})$, we have from the equation of $u$ 
\[
u_t\in L^{p'}(a,b; W^{-1, p'}(\mathbb{R}))+ L^1(\mathbb{R}\times(0, T)).
\]
Then, a local argument of Lemma \ref{Lemcompactness} yields the claim $(\ref{plap73b})$.
(Note that the last conclusion does not implies $u\in\mathcal{C}([0,\infty); L^1(\mathbb{R}))$ since the proof of Theorem $1.1$, \cite{Porretta} depends on the boundedness of $\Omega$. Moreover,   the proof of $(\ref{plap47a})$ is not applicable to prove $(\ref{plap73b})$, since the solution $u$ of the Cauchy problem is constructed in a different way)
\\

Now, to prove $u\in \mathcal{C}([0,\infty); L^1(\mathbb{R}))$, it suffices to show that $u(t)$ is continuous at $t=0$ in $L^1(\mathbb{R})$, i.e
\[\displaystyle\lim_{t\rightarrow 0}\int_{\mathbb{R}} |u(x,t)-u_0(x)| dx=0,\]
and the conclusion for $t>0$ is  proved  in the same way. In fact, we have for any $m\geq 1$
\begin{align*}
&\int_{\mathbb{R}} |u(x,t)-u_0(x)| dx\leq \int_{I_m} |u(x,t)-u_0(x)| dx  + \int_{\mathbb{R}\backslash I_m} |u(x,t)-u_0(x)| dx  
\\
& \leq\int_{I_m} |u(x,t)-u_0(x)| dx + \int_{\mathbb{R}\backslash I_m} u(x,t) dx +\int_{\mathbb{R}\backslash I_m} u_0(x) dx  =
\\
&    \int_{I_m} |u(x,t)-u_0(x)| dx -\left(\int_{I_m}( u(x,t)-u_0(x)) dx\right)   + \int_{\mathbb{R}} u(x,t) dx-  \int_{I_m}u_0(x) dx +\int_{\mathbb{R}\backslash I_m} u_0(x) dx
\\
&\overset{(\ref{plap71f})}{\leq} 2\int_{I_m} |u(x,t)-u_0(x)| dx+ \int_{\mathbb{R}} u_0(x) dx-  \int_{I_m}u_0(x) dx +\int_{\mathbb{R}\backslash I_m} u_0(x) dx  =
\\
& 2\int_{I_m} |u(x,t)-u_0(x)| dx+ 2\int_{\mathbb{R}\backslash I_m} u_0(x) dx. 
\end{align*}
 Taking $\displaystyle\limsup_{t\rightarrow0}$ both sides of the indicated inequality deduces
\[
\displaystyle\limsup_{t\rightarrow0} \int_{\mathbb{R}} |u(x,t)-u_0(x)| dx \leq  2\displaystyle\limsup_{t\rightarrow0} \int_{I_m} |u(x,t)-u_0(x)| dx+ 
2\int_{\mathbb{R}\backslash I_m} u_0(x) dx.
\]
By $u\in \mathcal{C}([0,\infty); L^1_{loc}(\mathbb{R}))$, we obtain from the last inequality
\[
\displaystyle\limsup_{t\rightarrow0} \int_{\mathbb{R}} |u(x,t)-u_0(x)| dx \leq 2\int_{\mathbb{R}\backslash I_m} u_0(x) dx.
\]
Then the result follows as $m\rightarrow\infty$.   Or, we complete the proof of Theorem $\ref{theexistenceCauchy}$.
\end{proof}
\begin{remark}
It is obvious that $u\in\mathcal{C}([0,\infty); L^q(\mathbb{R}))$, for any $q\geq 1$. Thus,
$u(0)=u_0$ in $L^q(\mathbb{R})$.
\end{remark}
\bigskip

Next, we show that the quenching phenomenon still holds for any weak nonnegative solutions of the Cauchy problem $(\ref{plapCauchy})$.
\begin{theorem}\label{theextinctionCauchy}
Let $v$ be such a solution of problem $(\ref{plapCauchy})$. Then, $v$ must vanish identically after a finite time $T_0>0$. Moreover,  $T_0$ can be estimated by  $\frac{\|u_0\|^{1+\beta}_{L^\infty(\mathbb{R})}}{1+\beta}$. 
\end{theorem}
\begin{proof}
It is not difficult to observe that 
\begin{equation}\label{plap73a}
\left\{
\begin{array}
[c]{lr}%
\partial_{t}v-(|v_{x}|^{p-2}v_{x})_{x}+ g_\varepsilon(v) + f(v)\psi_\varepsilon(v)\leq 0, &
\text{in}\hspace{0.05in}\mathbb{R}\times(0,\infty),
\\
v(x,0)=u_{0}(x), & \hspace{0.05in}\text{in}\hspace{0.05in}\mathbb{R}.
\end{array}
\right.  
\end{equation}
Remind that $g_\varepsilon(.)$ is a global Lipschitz function, while $f(.)\psi_\varepsilon(.)$ is a non-decreasing function. These facts allow us to apply  the strong comparison principle in order to get
\[
v(x,t)  \leq \Gamma_\varepsilon(t), \quad \forall (x,t)\in \mathbb{R}\times(0,\infty), 
\]
where $\Gamma_\varepsilon$ is in $(\ref{plap57a})$ with initial data $\|u_0\|_{L^\infty(\mathbb{R})}$.
\\
Letting $\varepsilon\rightarrow 0$ deduces
\[
v(x,t)  \leq \Gamma(t)=\left(\|u_0\|^{1+\beta}_{L^\infty(\mathbb{R})}-(1+\beta)t\right)_{+}^{\frac{1}{1+\beta}}  , \quad \forall (x,t)\in \mathbb{R}\times(0,\infty).
\]
This completes the proof of Theorem \ref{theextinctionCauchy}.
\end{proof}

\subsection{Existence of a maximal solution with compact support initially}
\hspace{0.2in} In general, we have no answer for  the existence of a maximal solution of the Cauchy problem. However, we will show that the solution $u$, constructed in Theorem \ref{theexistenceCauchy} is a maximal solution if the initial data has compact support.

\begin{theorem}\label{themaximalsol}
Assume that  $Supp(u_0)\subset I_{R_0}$. Then,  the solution $u$ constructed as in Theorem \ref{theexistenceCauchy} is a maximal solution of equation $(\ref{plapCauchy})$. Moreover, $Supp\left(u(t)\right)$ is bounded for all $t>0$.
\end{theorem}
\begin{proof} First, we have the following Lemma, which refers to  the finite speed of propagation of nonnegative solutions.
\begin{lemma}\label{lemsupp} Let $v$ be a weak solution of equation $(\ref{plapCauchy})$. Then,  $v$ has compact support at all later time $t>0$. Moreover, we have 
\[
Supp \left( v(t)\right) \subset I_{m_0}, \quad\text{for any } t>0, 
\]
with $m_0=R_{0} + \displaystyle\frac{ \|u_0\|_{L^\infty(\mathbb{R})}^{\frac{1}{\gamma}} } {\sigma}$, and  $\sigma=(\displaystyle\frac{1}{\gamma^{p-1}(\gamma-1)(p-1)})^{\frac{1}{p}}$.
\end{lemma}
\begin{proof}
For any $\varepsilon>0$, let $w_{\varepsilon}$ be a nonnegative solution of
the following equation
\begin{equation}
\left\{
\begin{array}
[c]{lr}%
-\partial_x \left(|\partial_x w_{\varepsilon}|^{p-2} \partial_x w_{\varepsilon}\right)+g_{\varepsilon}(w_{\varepsilon})=0, & \text{in}\hspace{0.05in}\mathbb{R}%
^{+},\\
w_{\varepsilon}(0)=\|u_{0}\|_{L^\infty(\mathbb{R})}, \\
\displaystyle\lim_{x\rightarrow\infty}w_{\varepsilon}(x)=0. &
\end{array}
\right.  \label{plap74}%
\end{equation}
It is straight forward  that
\[
w_{\varepsilon}(x)\overset{\varepsilon\rightarrow 0}{\longrightarrow} w(x)=\left(  \|u_{0}\|_{L^\infty}^{\frac{1}{\gamma}}%
-\sigma .x\right)  _{+}^{\gamma},\quad\text{for }x>0.
\]
 If we can  show
that
\begin{equation}
v(x,t)\leq w(x-R_{0}),\quad\text{for }x> R_{0},\hspace{0.05in}t>0,
\label{plap75}
\end{equation}
then $v(x,t)=0$, for any $x\geq  m_0$, and for $t>0$. The same argument for the case $x<-R_{0}$ implies $v(x,t)=0$, for any $x\leq - m_0$, and for $t>0$, thereby proves the above Lemma.
\\

Now, we prove $(\ref{plap75})$. Recall that $v$ satisfies $(\ref{plap73a})$ in $(R_0, \infty)\times(0,\infty)$.   Moreover, we have
\[
\left\{
\begin{array}
[c]{l}%
 v (x,t)\mid_{x=R_{0}} \leq \|u_{0}\|_{L^\infty} = w_{\varepsilon}(x-R_{0}%
)\mid_{x=R_{0}},\\
\\
v(x,0)=0\leq w_{\varepsilon}(x-R_{0}),\quad\text{for }x> R_{0}.
\end{array}
\right.
\]
By comparison principle, we obtain
\[
v(x,t)\leq w_\varepsilon (x), \quad \text{for } (x, t)\in(R_0, \infty)\times(0,\infty).
\]
Letting $\varepsilon\rightarrow 0$ yields conclusion $(\ref{plap75})$. This puts an end to the proof of Lemma \ref{lemsupp}.
\end{proof}
\\

It suffices to prove that $u$ is a maximal solution of problem $(\ref{plapCauchy})$. Indeed, let $v$ be a weak solution of problem $(\ref{plapCauchy})$. Thanks to Lemma \ref{lemsupp}, 
$v\displaystyle\mid_{I_{m_0}\times(0,\infty)}$ is a  weak solution of the following problem
\[
\left\{
\begin{array}
[c]{lr}%
\partial_{t}v-(|v_{x}|^{p-2}v_{x})_{x}+v^{-\beta}\chi_{\{v>0\}}   + f(v)=0 &
\text{in}\hspace{0.05in}I_{m_0}\times(0,\infty),\\
v(-m_0,t)=v(m_0,t)=0 & t\in(0,\infty),\\
v(x,0)=u_{0}(x) & \hspace{0.05in}\text{in}\hspace{0.05in}I_{m_0}.
\end{array}
\right.  
\]
This implies that $v(x,t)\leq u_r(x,t)$, in $\mathbb{R}\times(0,\infty)$, for any $r\geq m_0$. Passing $r\rightarrow\infty$  completes the proof of Theorem \ref{themaximalsol}. 
\end{proof}
\begin{remark}
Thanks to  Lemma \ref{lemsupp},  we observe that $u_r=u$, in $\mathbb{R}\times(0,\infty)$, for any $r\geq m_0.$ Thus, considering the Cauchy problem $(\ref{plapCauchy})$  is equivalent to considering Dirichlet problem $(\ref{plap1})$  in the case of compact support initially.
\end{remark}

\subsection{Instantaneous shrinking of compact support}
\hspace{0.2in} In this section, we will show that if $f$ satisfies a certain  growth condition at infinity, see $(H_3)$,  the ISS phenomenon occurs then for any nonnegative solution of equation $(\ref{plapCauchy})$. 
It is of course that there are many functions  satisfying either   $(H_1)$ and $(H_3)$, or  $(H_2)$ and $(H_3)$.  We can take for example: $f(s)=s^q$, for some $q>0$; or  $f(s)=e^s-1$; and so forth. After that, we have the following result
\begin{theorem}\label{theISS}
Let $f$ satisfy either $(H_1)$ and $(H_3)$, or  $(H_2)$ and $(H_3)$. Assume that $u_0\in L^1(\mathbb{R})\cap L^\infty(\mathbb{R})$, and $u_0(x)$ tends to $0$ uniformly as $|x|\rightarrow\infty$. Then any nonnegative solution of equation $(\ref{plapCauchy})$ has ISS property.
\end{theorem}
\begin{proof}
Let $v$ be a solution of equation $(\ref{plapCauchy})$. 
From $(H_3)$, there is a real number $R_0>0$ large enough such that $f(s)\geq s^{q_0}$, for $s\geq R_0$. Thus, we have 
\[
v^{-\beta}\chi_{\{v>0\}} + f(v)  \geq    R^{-(\beta+q_0)}_0 . v^{q_0}, 
\]
which leads to
\[
\partial_t v  -  (|v_x|^{p-2}v_x)_x  + R^{-(\beta+q_0)}_0 . v^{q_0} \leq 0, \quad \text{in }\mathbb{R}\times(0,\infty).
\]
Let $y$ be a unique solution of the following problem 
\begin{equation*}
\left\{
\begin{array}
[c]{lr}%
\partial_{t} y-(|y_{x}|^{p-2}y_{x})_{x}+  R^{-(\beta+q_0)}_0 . y^{q_0}=0, &
\text{in}\hspace{0.05in}\mathbb{R}\times(0,\infty),
\\
y(x,0)=u_{0}(x), & \hspace{0.05in}\text{in}\hspace{0.05in}\mathbb{R},
\end{array}
\right.  
\end{equation*}
(see e.g, \cite{Herrero}, \cite{WuJingHui},  \cite{Zhao}, and \cite{Dibe}).  By the strong comparison principle, we get
\[
v(x,t)\leq y(x,t),\quad \text{in }\mathbb{R}\times(0,\infty).
\]
Moreover, $y$ has the ISS property, see \cite{Herrero}, so does $v$. This puts an end to the proof of the above Theorem. 
\end{proof}
\begin{remark}\label{remfinal}
We also note that the result of  Theorem $\ref{theextinctionCauchy}$, Theorem $\ref{themaximalsol}$, and Theorem $\ref{theISS}$ still hold for the case where $f$ is a 
global Lipschitz function, and $f(0)=0$. 
\end{remark}

\section{Appendix}
{\bf Proof of Lemma \ref{uniqueness}}:
A subtraction between two equations satisfied by $v_1$ and $v_2$ gives us
\[
\partial_t(v_1-v_2) - \partial_x  \left( | \partial_x v_1|^{p-2}\partial_x v_1 -  | \partial_x v_2|^{p-2}\partial_x v_2\right) + g_{\varepsilon}(v_1)- g_{\varepsilon}(v_2) + f\psi_\varepsilon(v_1) -  f\psi_\varepsilon(v_2) \leq 0.
\] 
Multiplying both sides of the above equation with the test function $T_1(w)$, $w=  (v_1-v_2)_+$; and using integration by parts yield 
\begin{align*}
&\int_{I} S_1 (w(x,t)) dx + \int^{t}_{\tau}\int_{I} \left( | \partial_x v_1|^{p-2}\partial_x v_1 -  | \partial_x v_2|^{p-2}\partial_x v_2 \right)\left(\partial_x T_1(w)\right) dxds  +
\\
 & \int^{t}_{\tau}\int_{I}\left( g_{\varepsilon}(v_1)- g_{\varepsilon}(v_2)\right). T_1(w) dxds  + \int^{t}_{\tau}\int_{I}\left(f\psi_\varepsilon(v_1) -  f\psi_\varepsilon(v_2) \right). T_1(w) dxds\leq \int_{I} S_1 (w(x,\tau)) dx,
\end{align*}
for $t>\tau>0$.
It follows from the monotone of $p$-Laplacian, and the monotone of $f\psi_\varepsilon$, 
and the fact that $g_\varepsilon$ is a global Lipschitz function
\[
\int_{I} S_1 (w(x,t)) dx  \leq C(\varepsilon) \int^{t}_{\tau}\int_{I} |v_1-v_2|T_1(w) dxds+ \int_{I} S_1 (w(x,\tau)) dx,
\]
where $C(\varepsilon)>0$ is the Lipschitz constant of $g_\varepsilon$. Letting $\tau\rightarrow0$ in the above inequality deduces
\[
\int_{I} S_1 (w(x,t)) dx  \leq C(\varepsilon) \int^{t}_{0}\int_{I} |v_1-v_2|T_1(w) dxds.
\]
\\
In addition, we have
\[
|v_1-v_2|T_1(w)(x,t)  \leq 2  S_1 (w(x,t)).
\] 
Inserting this fact into the indicated inequality yields
\[
\int_{I} S_1 (w(x,t)) dx  \leq 2C(\varepsilon) \int^{t}_{0}\int_{I}  S_1 (w(x,t)) dxds.
\]
Then, we arrive to the following ordinary differential equation 
\[
\left\{
\begin{array}
[c]{lr}%
\frac{d}{dt} y(t) \leq 2C(\varepsilon)y(t), \quad  t>0,  \\
\\
y(0)=0.
\end{array}
\right.
\]
with \[y(t)=  \int_{I} S_1(w(x,t))dx .\]
It follows from Gronwall's lemma that
\[
y(t)=0,\quad \forall t>0,
\]
which implies 
\[
w(t)=0, \quad\forall t>0.
\] 
In other words, we get the above lemma.
\begin{remark} The result of Lemma \ref{uniqueness} also holds for  any sub-solution $v_1$ and super-solution $v_2$ of equation $(\ref{plap17})$ satisfying $v_2 \geq v_1$ on the boundary.
\end{remark}

\end{document}